\documentclass{amsart}
\usepackage[psamsfonts]{amssymb}
\usepackage{amsthm,amsmath}
\usepackage{graphicx}

\title[Solving Chisini's functional equation]{Solving Chisini's functional equation}

\author{Jean-Luc Marichal}
\address{Mathematics Research Unit, FSTC, University of Luxembourg \\
6, rue Coudenhove-Kalergi, L-1359 Luxembourg, Luxembourg} \email{jean-luc.marichal[at]uni.lu}

\date{March 17, 2010}

\begin{document}

\theoremstyle{plain}
\newtheorem{theorem}{Theorem}[section]
\newtheorem{lemma}[theorem]{Lemma}
\newtheorem{proposition}[theorem]{Proposition}
\newtheorem{corollary}[theorem]{Corollary}
\newtheorem{fact}[theorem]{Fact}

\theoremstyle{definition}
\newtheorem{definition}[theorem]{Definition}
\newtheorem{example}[theorem]{Example}

\theoremstyle{remark}
\newtheorem{conjecture}[theorem]{Conjecture}
\newtheorem{remark}[theorem]{Remark}

\newcommand{\N}{\mathbb{N}}
\newcommand{\R}{\mathbb{R}}
\newcommand{\I}{\mathbb{I}}
\newcommand{\J}{\mathbb{J}}
\newcommand{\F}{\mathsf{F}}
\newcommand{\G}{\mathsf{G}}
\newcommand{\M}{\mathsf{M}}
\newcommand{\U}{\mathsf{U}}
\newcommand{\Vspace}{\vspace{2ex}}

\newcommand{\ran}{\mathrm{ran}}
\newcommand{\dom}{\mathrm{dom}}
\newcommand{\id}{\mathrm{id}}
\newcommand{\bfx}{\mathbf{x}}
\newcommand{\bfy}{\mathbf{y}}
\newcommand{\bfz}{\mathbf{z}}
\newcommand{\bfe}{\mathbf{e}}

\newcommand{\Min}{\mathsf{Min}}
\newcommand{\Max}{\mathsf{Max}}
\newcommand{\Med}{\mathsf{Med}}
\def\relstack#1#2{\mathrel{\mathop{#1}\limits_{#2}}}

\begin{abstract}
We investigate the $n$-variable real functions $\G$ that are solutions of the Chisini functional equation
$\F(\bfx)=\F(\G(\bfx),\ldots,\G(\bfx))$, where $\F$ is a given function of $n$ real variables. We provide necessary and sufficient conditions on
$\F$ for the existence and uniqueness of solutions. When $\F$ is nondecreasing in each variable, we show in a constructive way that if a
solution exists then a nondecreasing and idempotent solution always exists. We also provide necessary and sufficient conditions on $\F$ for the
existence of continuous solutions and we show how to construct such a solution. We finally discuss a few applications of these results.
\end{abstract}



\maketitle

\section{Introduction}

Let $\I$ be any nonempty real interval, bounded or not, and let $\F\colon\I^n\to\R$ be any given function. We are interested in the solutions
$\G\colon\I^n\to\I$ of the following functional equation
\begin{equation}\label{eq:FuncEqFG}
\F(x_1,\ldots,x_n)=\F(\G(x_1,\ldots,x_n),\ldots,\G(x_1,\ldots,x_n)).
\end{equation}

This functional equation was implicitly considered in 1929 by Chisini~\cite[p.~108]{Chi29}, who investigated the concept of mean as an
\emph{average}\/ or a \emph{numerical equalizer}. More precisely, Chisini defined a mean of $n$ numbers $x_1,\ldots,x_n\in\I$ with respect to a
function $\F\colon\I^n\to\R$ as a number $M$ such that
$$
\F(x_1,\ldots,x_n)=\F(M,\ldots,M).
$$
For instance, when $\I=\left]0,\infty\right[$ and
$\F$ is the sum, the product, the sum of squares, the sum of inverses, or the sum of exponentials, the solution $M$ of the equation above is
unique and consists of the arithmetic mean, the geometric mean, the quadratic mean, the harmonic mean, and the exponential mean, respectively.

By considering the \emph{diagonal section}\/ of $\F$, i.e., the one-variable function $\delta_{\F}\colon\I\to\R$ defined by
$\delta_{\F}(x):=\F(x,\ldots,x)$, we can rewrite equation (\ref{eq:FuncEqFG}) as
\begin{equation}\label{eq:FuncEqFG2}
\F=\delta_{\F}\circ\G.
\end{equation}
If, as in the examples above, we assume that $\F$ is nondecreasing (in each variable) and that $\delta_{\F}$ is a bijection from $\I$ onto the
range of $\F$, then Chisini's equation (\ref{eq:FuncEqFG2}) clearly has a unique solution $\G=\delta_{\F}^{-1}\circ\F$ which is nondecreasing
and idempotent (i.e., such that $\delta_{\G}(x)=x$). Such a solution is then called a \emph{Chisini mean} or a \emph{level surface mean} (see
Bullen\ \cite[VI.4.1]{Bul03}).

In this paper, we consider Chisini's functional equation (\ref{eq:FuncEqFG2}) in its full generality, i.e., without any assumption on $\F$. We
first provide necessary and sufficient conditions on $\F$ for the existence of solutions and we show how the possible solutions can be
constructed (Section 2). We also investigate solutions of the form $g\circ\F$, where $g$ is a quasi-inverse of $\delta_{\F}$
(Section 3). We then elaborate on the case when $\F$ is nondecreasing and we show that if a solution exists then at least one nondecreasing and
idempotent solution always exists. We construct such a solution by means of a metric interpolation (inspired from Urysohn's lemma and Shepard's
interpolation method) and we discuss some of its properties (Section 4). We also show that this solution obtained by interpolation is continuous
whenever a continuous solution exists and we provide necessary and sufficient conditions for the existence of continuous solutions (Section 5).
Surprisingly enough, continuity of $\F$ is neither necessary nor sufficient to ensure the existence of continuous solutions. Finally, we discuss
a few applications of the theory developed here to certain classes of functions (Section 6). In particular, we revisit the concept of Chisini mean
and we extend it to the case when $\delta_{\F}$ is nondecreasing but not strictly increasing.

The terminology used throughout this paper is the following. For any integer $n\geqslant 1$, we set $[n]:=\{1,\ldots,n\}$. The domain and range
of any function $f$ are denoted by $\dom(f)$ and $\ran(f)$, respectively. The minimum and maximum functions are denoted by $\Min$ and $\Max$,
respectively. That is,
$$
\Min(\bfx):=\min\{x_1,\ldots,x_n\}\quad\mbox{and}\quad\Max(\bfx):=\max\{x_1,\ldots,x_n\}
$$
for any $\bfx\in\R^n$. The identity
function is the function $\id\colon\R\to\R$ defined by $\id(x)=x$. For any $i\in [n]$, $\mathbf{e}_i$ denotes the $i$th unit vector of $\R^n$.
We also set $\mathbf{1}:=(1,\ldots,1)\in\R^n$. The diagonal restriction of a subset $S\subseteq\I^n$ is the subset ${\rm diag}(S):=\{x\mathbf{1}
: x\in\I\}\cap S$. Finally, inequalities between vectors in $\R^n$, such as $\bfx\leqslant\bfx'$, are understood componentwise.

\section{Resolution of Chisini's equation}

In this section we provide necessary and sufficient conditions for the existence and uniqueness of solutions of Chisini's equation and we show
how the solutions can be constructed.

Let $\F\colon\I^n\to\R$ be a given function and suppose that the associated Chisini equation (\ref{eq:FuncEqFG2}) has a solution
$\G\colon\I^n\to\I$. We immediately see that, for any $\bfx\in\I^n$, the possible values of $\G(\bfx)$ are exactly those reals $z\in\I$ for
which the $n$-tuple $(z,\ldots,z)$ belongs to the level set of $\F$ through $\bfx$. In other terms, we must have
\begin{equation}\label{eq:GindFcF}
\G(\bfx)\in\delta_{\F}^{-1}\{\F(\bfx)\},\qquad\forall \bfx\in\I^n.
\end{equation}
Thus a necessary condition for equation (\ref{eq:FuncEqFG2}) to have at least one solution is
\begin{equation}\label{eq:ranFrandF}
\ran(\delta_{\F})=\ran(\F).
\end{equation}
This fact also follows from the following sequence of inclusions:
$\ran(\delta_{\F})\subseteq\ran(\F)=\ran(\delta_{\F}\circ\G)\subseteq\ran(\delta_{\F})$.

Assuming the Axiom of Choice (AC), we immediately see that condition (\ref{eq:ranFrandF}) is also sufficient for equation (\ref{eq:FuncEqFG2})
to have at least one solution. Indeed, by assuming both AC and (\ref{eq:ranFrandF}), we can define a function $\G\colon\I^n\to\I$ satisfying
(\ref{eq:GindFcF}) and this function then solves equation (\ref{eq:FuncEqFG2}). Note however that AC is not always required to ensure the
existence of a solution. For instance, if $\delta_{\F}$ is monotonic (i.e., either nondecreasing or nonincreasing), then every level set
$\delta_{\F}^{-1}\{\F(\bfx)\}$ is a bounded interval (except two of them at most) and for instance its midpoint could be chosen to define
$\G(\bfx)$.

Thus we have proved the following result.

\begin{proposition}\label{prop:FdFGrr}
Let $\F\colon\I^n\to\R$ be a function. If equation (\ref{eq:FuncEqFG2}) has at least one solution $\G\colon\I^n\to\I$ then
$\ran(\delta_{\F})=\ran(\F)$. Under AC (not necessary if $\delta_{\F}$ is monotonic), the converse also holds.
\end{proposition}

The following example shows that condition (\ref{eq:ranFrandF}) does not hold for every function $\F$, even if $\F$ is nondecreasing.

\begin{example}\label{ex:NilMin}
The \emph{nilpotent minimum}\/ (see e.g.\ \cite{KleMes05}) is the function $\mathsf{T}^{\mathrm{nM}}\colon [0,1]^2\to [0,1]$ defined as
$$
\mathsf{T}^{\mathrm{nM}}(x_1,x_2):=
\begin{cases}
0,              & \mbox{if $x_1+x_2\leqslant 1$,}\\
\Min(x_1,x_2),  & \mbox{otherwise.}
\end{cases}
$$
We clearly have $\ran(\mathsf{T}^{\mathrm{nM}})=[0,1]$ and $\ran(\delta_{\mathsf{T}^{\mathrm{nM}}})=\{0\}\cup\,]\frac 12,1]$ (see
Figure~\ref{fig:NilMin}), and hence the associated Chisini equation has no solution.
\end{example}

\begin{figure}[htb]
\includegraphics[height=.47\linewidth, keepaspectratio=true]{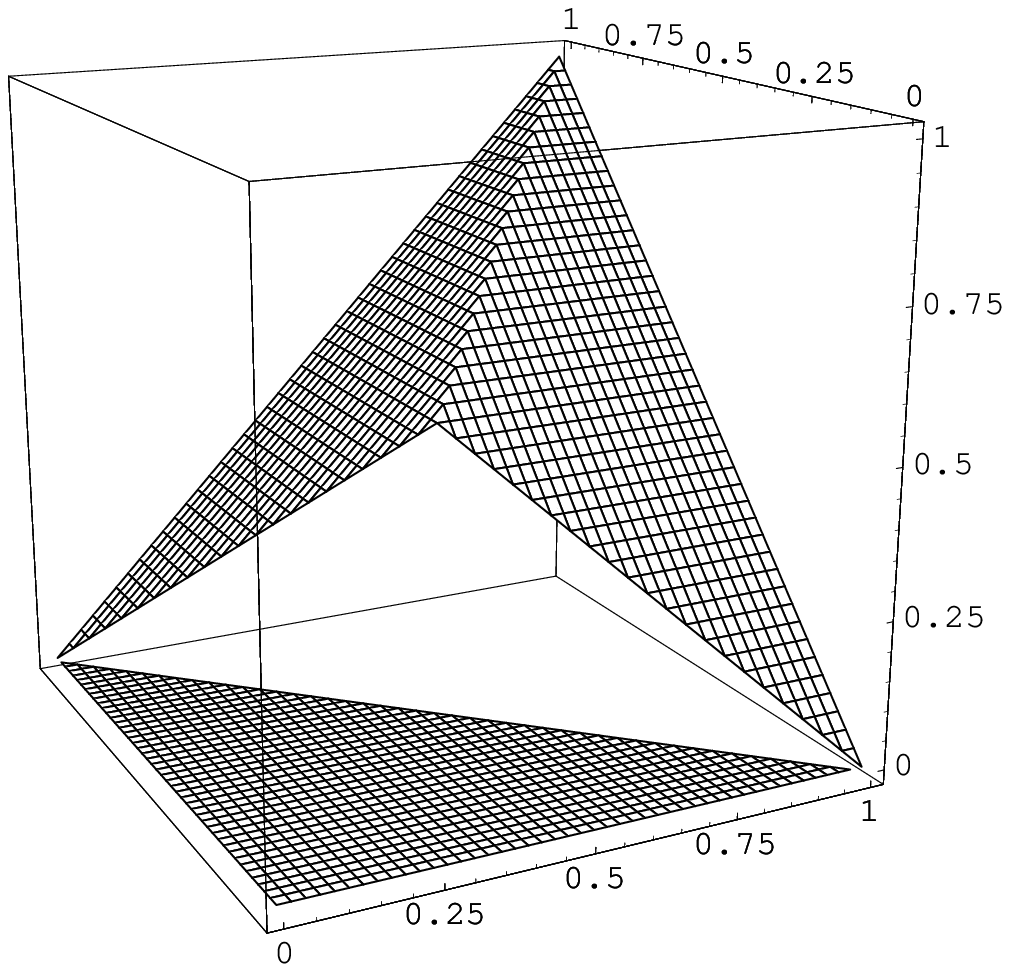}\hfill
\includegraphics[height=.45\linewidth, keepaspectratio=true]{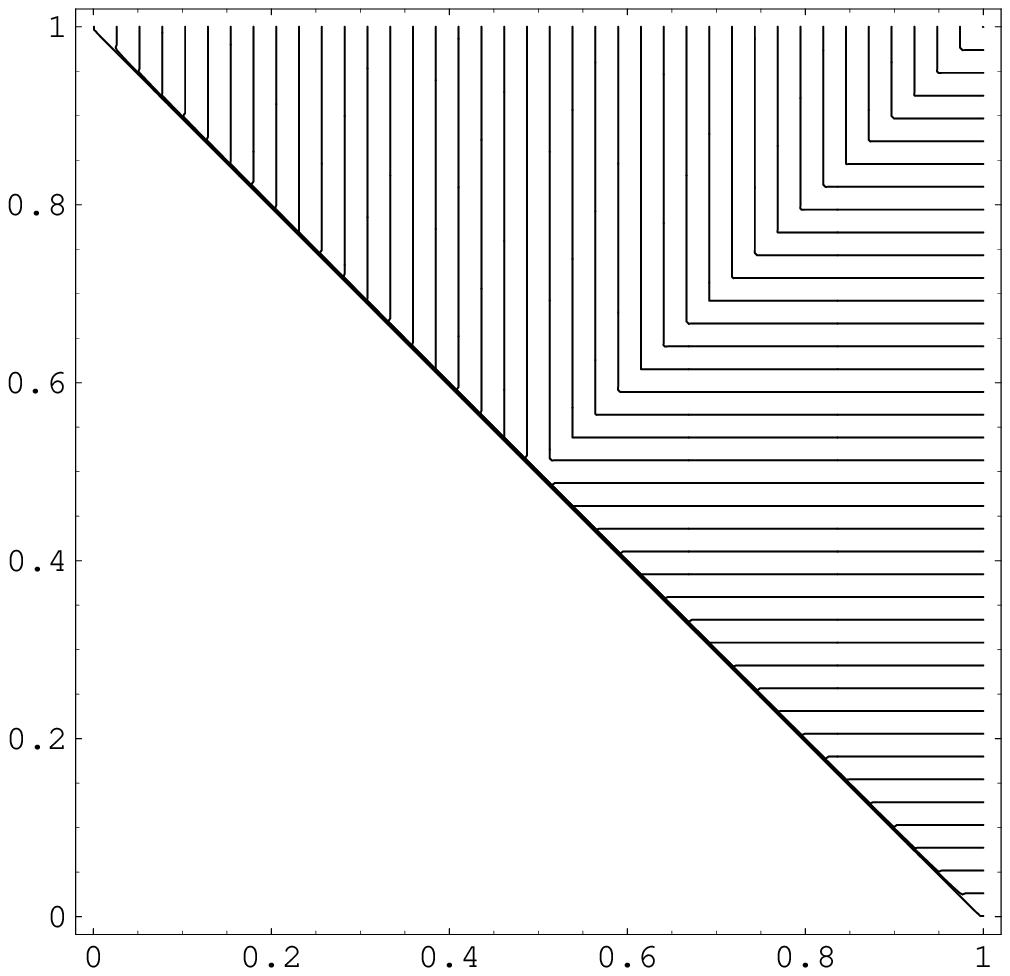}
\caption{Nilpotent minimum (3D plot and contour plot)} \label{fig:NilMin}
\end{figure}

Equation (\ref{eq:GindFcF}) shows how the possible solutions of the Chisini equation can be constructed. We first observe that $\I^n$ is a
disjoint union of the level sets of $\F$, i.e.,
$$
\I^n=\bigcup_{y\in\ran(\delta_{\F})}\F^{-1}\{y\}.
$$
Thus, constructing a solution $\G$
on $\I^n$ reduces to constructing it on each level set $\F^{-1}\{y\}$, with $y\in\ran(\delta_{\F})$. That is, for every $\bfx\in\F^{-1}\{y\}$,
we choose $\G(\bfx)\in\delta_{\F}^{-1}\{y\}$.

The next proposition yields an alternative description of the solutions of the Chisini equation through the concept of quasi-inverse function.
Recall first that a function $g$ is a {\em quasi-inverse}\/ of a function $f$ if
\begin{eqnarray}
&& f\circ g|_{\ran(f)}=\id|_{\ran(f)},\label{eq:defqi1}\\
&& \ran(g|_{\ran(f)})=\ran(g).\label{eq:defqi2}
\end{eqnarray}

For any function $f$, denote by $Q(f)$ the set of its quasi-inverses. This set is nonempty whenever we assume AC, which is actually just another
form of the statement ``every function has a quasi-inverse.'' Recall also that the relation of being quasi-inverse is symmetric, i.e., if $g\in
Q(f)$ then $f\in Q(g)$; moreover $\ran(f)\subseteq\dom(g)$ and $\ran(g)\subseteq\dom(f)$ (see \cite[Sect.~2.1]{SchSkl83}).

By definition, if $g\in Q(f)$ then $g|_{\ran(f)}\in Q(f)$. Thus we can always restrict the domain of any quasi-inverse $g\in Q(f)$ to $\ran(f)$.
These ``restricted'' quasi-inverses, also called \emph{right-inverses}, are then simply characterized by condition (\ref{eq:defqi1}), which can
be rewritten as
\begin{equation}\label{eq:defqi11}
g(y)\in f^{-1}\{y\},\qquad \forall y\in\ran(f).
\end{equation}

\begin{proposition}\label{prop:SolEqFDG}
Let $\F\colon\I^n\to\R$ be a function satisfying (\ref{eq:ranFrandF}) and let $\G\colon\I^n\to\I$ be any function. Then, assuming AC (not
necessary if $\delta_{\F}$ is monotonic), the following assertions are equivalent:
\begin{enumerate}
\item[(i)] We have $\F=\delta_{\F}\circ\G$.

\item[(ii)] For any $\bfx\in\I^n$, we have $\G(\bfx)\in\delta_{\F}^{-1}\{\F(\bfx)\}$.

\item[(iii)] For any $\bfx\in\I^n$, there is $g_{\bfx}\in Q(\delta_{\F})$ such that $\G(\bfx)=(g_{\bfx}\circ\F)(\bfx)$.
\end{enumerate}
\end{proposition}

\begin{proof}
The implications $(iii)\Rightarrow (ii)\Rightarrow (i)$ are immediate. Let us prove that $(i)\Rightarrow (iii)$. Fix $\bfx\in\I^n$ and set
$y=\F(\bfx)$. We have $\G(\bfx)\in\delta_{\F}^{-1}\{y\}$ and, by (\ref{eq:defqi11}), there exists $g_{\bfx}\in Q(\delta_{\F})$ such that
$\G(\bfx)=g_{\bfx}(y)=(g_{\bfx}\circ\F)(\bfx)$.
\end{proof}

A necessary and sufficient condition for equation (\ref{eq:FuncEqFG2}) to have a unique solution immediately follows from the assertion (ii) of
Proposition~\ref{prop:SolEqFDG}.

\begin{corollary}\label{cor:uniqueness}
Let $\F\colon\I^n\to\R$ be a function satisfying (\ref{eq:ranFrandF}). Then, assuming AC (not necessary if $\delta_{\F}$ is monotonic), the
associated Chisini equation (\ref{eq:FuncEqFG2}) has a unique solution if and only if $\delta_{\F}$ is one-to-one. The solution is then given by
$\G=\delta_{\F}^{-1}\circ\F$.
\end{corollary}

The special case when $\F$ is a symmetric function of its variables is of particular interest. For instance, it is then easy to see that there
are always symmetric solutions of the Chisini equation. We now state a slightly more general (but immediate) result.

Let $\mathfrak{S}_n$ be the set of permutations on $[n]$ and let $\F\colon\I^n\to\R$ be any function. We say that $\sigma\in\mathfrak{S}_n$ is a
\emph{symmetry} of $\F$ if $\F([\bfx]_{\sigma})=\F(\bfx)$ for every $\bfx\in\I^n$, where $[\bfx]_{\sigma}$ denotes the $n$-tuple
$(x_{\sigma(1)},\ldots,x_{\sigma(n)})$.

\begin{proposition}
Let $\F\colon\I^n\to\R$ be a function satisfying (\ref{eq:ranFrandF}) and let $\sigma\in\mathfrak{S}_n$ be a symmetry of $\F$. If
$\G\colon\I^n\to\I$ is a solution of Chisini's equation (\ref{eq:FuncEqFG2}), then the function $\G_{\sigma}\colon\I^n\to\I$, defined by
$\G_{\sigma}(\bfx)=\G([\bfx]_{\sigma})$, is also a solution of (\ref{eq:FuncEqFG2}).
\end{proposition}

\section{Quasi-inverse based solutions}
\label{sec:TrivialS}

In this section, we investigate special solutions whose construction is inspired from Proposition~\ref{prop:SolEqFDG} (iii). These solutions are
described in the following immediate result.

\begin{proposition}\label{prop:FdFGrr2}
Let $\F\colon\I^n\to\R$ be a function satisfying condition (\ref{eq:ranFrandF}). Then, assuming AC (not necessary if $\delta_{\F}$ is
monotonic), for any $g\in Q(\delta_{\F})$, the function $\G=g\circ\F$, from $\I^n$ to $\I$, is well defined and solves Chisini's equation
(\ref{eq:FuncEqFG2}).
\end{proposition}

Proposition~\ref{prop:FdFGrr2} motivates the following definition. Given a function $\F\colon\I^n\to\R$ satisfying condition
(\ref{eq:ranFrandF}), we say that a function $\G\colon\I^n\to\I$ is a \emph{quasi-inverse based solution} (or \emph{$Q$-solution}) of Chisini's
equation (\ref{eq:FuncEqFG2}) if there exists $g\in Q(\delta_{\F})$ such that $\G=g\circ\F$.

Recall that a function $\F\colon\I^n\to\R$ is said to be \emph{idempotent}\/ if $\delta_{\F}=\id$. We say that $\F$ is \emph{range-idempotent}\/
if $\ran(\F)\subseteq \I$ and $\delta_{\F}|_{\ran(\F)}=\id|_{\ran(\F)}$, where the latter condition can be rewritten as $\delta_{\F}\circ\F=\F$.
We can readily see that any $Q$-solution $\G\colon\I^n\to\I$ of Chisini's equation (\ref{eq:FuncEqFG2}) is range-idempotent. Indeed, since
$\G=g\circ\F$ for some $g\in Q(\delta_{\F})$, we simply have $\delta_{\G}\circ\G=g\circ\delta_{\F}\circ\G=g\circ\F=\G$.

An interesting feature of $Q$-solutions $\G$ is that, in addition of being range-idempotent, they may inherit certain properties from $\F$, such as
nondecreasing monotonicity, symmetry, continuity, etc. For instance, $\sigma\in\mathfrak{S}_n$ is a symmetry of $\G$ if and only if it is a
symmetry of $\F$. Also, $\G$ is nondecreasing as soon as $\F$ is either nondecreasing or nonincreasing. The latter result follows from the
fact that if a function $f\colon\I\to\R$ is nondecreasing (resp.\ nonincreasing) then so is every $g\in Q(f)$; see \cite[Sect.~4.4]{SchSkl83}.
However, as the following example shows, Chisini's equation may have non-$Q$-solutions and the $Q$-solutions may be non-idempotent.

\begin{example}\label{ex:Luka}
The \emph{{\L}ukasiewicz t-norm}\/ (see e.g.\ \cite{KleMes05}) is the function $\mathsf{T}^{\mathrm{L}}\colon [0,1]^2\to [0,1]$ defined as
$$
\mathsf{T}^{\mathrm{L}}(x_1,x_2):=\Max(0,x_1+x_2-1).
$$
We have $\delta_{\mathsf{T}^{\mathrm{L}}}(x)=\Max(0,2x-1)$ and any $g\in
Q(\delta_{\mathsf{T}^{\mathrm{L}}})$ is such that $g(x)=\frac 12(x+1)$ on $]0,1]$ and $g(0)\in [0,\frac 12]$. Thus, no function of the form
$g\circ\mathsf{T}^{\mathrm{L}}$ is idempotent on $[0,1]^2$. However, the idempotent function $\G(x_1,x_2)=\frac 12(x_1+x_2)$ clearly solves the
Chisini equation $\mathsf{T}^{\mathrm{L}}=\delta_{\mathsf{T}^{\mathrm{L}}}\circ\G$.
\end{example}

The $Q$-solutions of Chisini's equation can be easily transformed into idempotent solutions. Indeed, for any $g\in Q(\delta_{\F})$, the function
$\G\colon\I^n\to\I$, defined by
$$
\G(\bfx)=
\begin{cases}
x_1, & \mbox{if $\bfx\in\mathrm{diag}(\I^n)$},\\
(g\circ\F)(\bfx), & \mbox{otherwise},
\end{cases}
$$
is an idempotent solution. However, for such solutions, some properties of the $Q$-solutions, such as nondecreasing monotonicity, might be lost.

This motivates the natural question whether the Chisini equation, when solvable, has nondecreasing and idempotent solutions. In the next
section, we show in a constructive way that, if $\F$ is nondecreasing and satisfies condition (\ref{eq:ranFrandF}), at least one such solution
always exists.

\section{Nondecreasing and idempotent solutions}

We now examine the situation when $\F$ is nondecreasing, in which case condition (\ref{eq:ranFrandF}) alone ensures the solvability of Chisini's
equation. Clearly $\delta_{\F}$ is then nondecreasing and hence its level sets $\delta_{\F}^{-1}\{y\}$, $y\in\ran(\delta_{\F})$, are intervals.
It follows that $\delta_{\F}$ always has a nondecreasing quasi-inverse $g\in Q(\delta_{\F})$ (without an appeal to AC) and hence the
$Q$-solution $\G=g\circ\F$ is also nondecreasing and even range-idempotent (see Section~\ref{sec:TrivialS}). However, as we observed in
Example~\ref{ex:Luka}, this solution need not be idempotent.

In this section we show that, assuming condition (\ref{eq:ranFrandF}), at least one nondecreasing and idempotent solution always exists and we
show how to construct such a solution (see Theorem~\ref{thm:GF-ND-Id}). Roughly speaking, the idea consists in constructing on each level set
$\F^{-1}\{y\}$, for $y\in\ran(\delta_{\F})$, a nondecreasing and idempotent function that assumes the value $\inf\delta_{\F}^{-1}\{y\}$ (resp.\
$\sup\delta_{\F}^{-1}\{y\}$) on the common edge of the level set $\F^{-1}\{y\}$ and the adjacent lower (resp.\ upper) level set. As we will
discuss in Remark~\ref{rem:GF-ND-Id} $(ii)$-$(iii)$, this construction actually consists of a metric interpolation based on both Urysohn's lemma
and Shepard's interpolation method.

Let $\F\colon\I^n\to\R$ be a nondecreasing function satisfying (\ref{eq:ranFrandF}). For any $y\in\ran(\delta_{\F})$, consider the corresponding
lower and upper level sets of $\F$, defined by
$$
\F_{<}^{-1}(y) := \{\bfx\in\I^n : \F(\bfx)<y\}\qquad\mbox{and}\qquad \F_{>}^{-1}(y) := \{\bfx\in\I^n : \F(\bfx)>y\},
$$
respectively. Consider the Chebyshev distance between two points $\bfx,\bfx'\in\R^n$ and between a point $\bfx\in\R^n$ and a subset
$S\subseteq\R^n$,
\begin{eqnarray*}
d_{\infty}(\bfx,\bfx') &:=& \|\bfx-\bfx'\|_{\infty} ~=~ \max_{i\in [n]} |x_i-x'_i|,\\
d_{\infty}(\bfx,S) &:=& \inf_{\bfx'\in S} \|\bfx-\bfx'\|_{\infty},
\end{eqnarray*}
with the convention that $d_{\infty}(\bfx,\varnothing)=\infty$. Define also the following functions, from $\I^n$ to $[-\infty,\infty]$,
$$
a_{\F}(\bfx) := \inf\delta_{\F}^{-1}\{\F(\bfx)\},\qquad b_{\F}(\bfx) := \sup\delta_{\F}^{-1}\{\F(\bfx)\},
$$
and
\begin{eqnarray*}
d_{\F}^<(\bfx) &:=& d_{\infty}\big(\bfx,\F_{<}^{-1}(\F(\bfx))\big) ~=~ \inf_{\textstyle{\bfx'\in\I^n\atop \F(\bfx')<\F(\bfx)}} \|\bfx-\bfx'\|_{\infty},\\
d_{\F}^>(\bfx) &:=& d_{\infty}\big(\bfx,\F_{>}^{-1}(\F(\bfx))\big) ~=~ \inf_{\textstyle{\bfx'\in\I^n\atop \F(\bfx')>\F(\bfx)}}
\|\bfx-\bfx'\|_{\infty}.
\end{eqnarray*}

The next lemma concerns the case when $d_{\F}^<(\bfx)=\infty$ (resp.\ $d_{\F}^>(\bfx)=\infty$), which means that
$\F_{<}^{-1}(\F(\bfx))=\varnothing$ (resp.\ $\F_{>}^{-1}(\F(\bfx))=\varnothing$).

\begin{lemma}\label{lemma:dab}
Let $\F\colon\I^n\to\R$ be a nondecreasing function satisfying (\ref{eq:ranFrandF}) and let $\bfx\in\I^n$. If $d_{\F}^<(\bfx)=\infty$ (resp.\
$d_{\F}^>(\bfx)=\infty$) then $a_{\F}(\bfx)=\inf\I$ (resp.\ $b_{\F}(\bfx)=\sup\I$). The converse holds if $\inf\I\notin\I$ (resp.\
$\sup\I\notin\I$).
\end{lemma}

\begin{proof}
We prove the lower bound statement only; the other one can be established dually. Let $\bfx\in\I^n$ and assume that $d_{\F}^<(\bfx)=\infty$,
which means that $\F(\bfx)\leqslant\F(\bfx')$ for all $\bfx'\in\I^n$. Then the result immediately follows for if there were $x\in\I$ such that
$x<a_{\F}(\bfx)$ then we would obtain $\delta_{\F}(x)<\F(\bfx)$, a contradiction. To prove the converse claim, assume that
$a_{\F}(\bfx)=\inf\I\notin\I$ and suppose that there is $\bfx'\in\I^n$ such that $\F(\bfx')<\F(\bfx)$. By nondecreasing monotonicity, we have
$$
(\delta_{\F}\circ\Min)(\bfx')\leqslant\F(\bfx')<\F(\bfx).
$$
But then we must have $\Min(\bfx')=\inf\I$ and hence $\bfx'\notin\I^n$, a
contradiction.
\end{proof}

\begin{remark}
In the second part of Lemma~\ref{lemma:dab}, the condition $\inf\I\notin\I$ (resp.\ $\sup\I\notin\I$) cannot be dropped off. Indeed, let e.g.\
$\F\colon [a,b]^n\to\R$ be defined by $\F(a\mathbf{1})=0$ and $\F(\bfx)=1$ if $\bfx\neq a\mathbf{1}$. Then, for any $\bfx\neq a\mathbf{1}$, we
have $a_{\F}(\bfx)=a$ but $d_{\F}^<(\bfx)<\infty$.
\end{remark}

Now, consider the following subdomain of $\I^n$:
$$
\Omega_{\F} := \{\bfx\in\I^n : d_{\F}^>(\bfx)+d_{\F}^<(\bfx)>0\}
$$
and define the function $\U_{\F}\colon\Omega_{\F}\to\R$ by
$$
\U_{\F}(\bfx) := \frac{d_{\F}^>(\bfx)\, a_{\F}^{\mathstrut}(\bfx)+d_{\F}^<(\bfx)\, b_{\F}^{\mathstrut}(\bfx)}{d_{\F}^>(\bfx)+d_{\F}^<(\bfx)}.
$$
By Lemma~\ref{lemma:dab}, we immediately observe that this function is well defined if and only if both $d_{\F}^<(\bfx)$ and $d_{\F}^>(\bfx)$
are bounded. By extension, when only $d_{\F}^>(\bfx)$ is bounded, we naturally consider the limiting value
$$
\U_{\F}(\bfx) ~:=~ \lim_{a\to -\infty} \frac{d_{\F}^>(\bfx)\, a+d_{\infty}(\bfx,a\mathbf{1})\, b_{\F}^{\mathstrut}(\bfx)}
{d_{\F}^>(\bfx)+d_{\infty}(\bfx,a\mathbf{1})} ~=~ b_{\F}^{\mathstrut}(\bfx)-d_{\F}^>(\bfx).
$$
Similarly, when only $d_{\F}^<(\bfx)$ is bounded, we consider the limiting value
$$
\U_{\F}(\bfx) ~:=~ \lim_{b\to +\infty} \frac{d_{\infty}(\bfx,b\mathbf{1})\, a_{\F}^{\mathstrut}(\bfx)+d_{\F}^<(\bfx)\, b}
{d_{\infty}(\bfx,b\mathbf{1})+d_{\F}^<(\bfx)} ~=~ a_{\F}^{\mathstrut}(\bfx)+d_{\F}^<(\bfx).
$$
Finally, when both $d_{\F}^>(\bfx)$ and $d_{\F}^<(\bfx)$ are unbounded (i.e., when $\F$ is a constant function), we consider
$$
\U_{\F}(\bfx) ~:=~ \lim_{b\to +\infty} \frac{d_{\infty}(\bfx,b\mathbf{1})(-b)+d_{\infty}(\bfx,-b\mathbf{1})\, b} {d_{\infty}(\bfx,b\mathbf{1})+
d_{\infty}(\bfx,-b\mathbf{1})} ~=~ \textstyle{\frac 12\Min(\bfx)+\frac 12\Max(\bfx)}.
$$
We now define the function $\M_{\F}\colon\I^n\to\R$ by
\begin{equation}\label{eq:Urysohn}
\M_{\F}(\bfx) :=
\begin{cases}
\U_{\F}(\bfx), & \mbox{if $\bfx\in\Omega_{\F}$},\\
\frac 12 a_{\F}(\bfx)+\frac 12 b_{\F}(\bfx), & \mbox{if $\bfx\in\I^n\setminus\Omega_{\F}$}.
\end{cases}
\end{equation}
Even though the function $\M_{\F}$ is well defined on $\I^n$, there are still situations in which this function needs to be slightly modified on certain level sets to
ensure the solvability condition $\M_{\F}(\bfx)\in\delta_{\F}^{-1}\{\F(\bfx)\}$ (see Proposition~\ref{prop:SolEqFDG}).

In fact, suppose there exists $\bfx^*\in\I^n$ such that
\begin{equation}\label{eq:Case1N}
a_{\F}(\bfx^*)\notin\delta_{\F}^{-1}\{\F(\bfx^*)\}~\mbox{and}~\exists\bfx\in\F^{-1}\{\F(\bfx^*)\}\cap\Omega_{\F}~\mbox{such
that}~d_{\F}^<(\bfx)=0,
\end{equation}
or
\begin{equation}\label{eq:Case2N}
b_{\F}(\bfx^*)\notin\delta_{\F}^{-1}\{\F(\bfx^*)\}~\mbox{and}~\exists\bfx\in\F^{-1}\{\F(\bfx^*)\}\cap\Omega_{\F}~\mbox{such
that}~d_{\F}^>(\bfx)=0.
\end{equation}
In either case, we replace the restriction of $\M_{\F}$ to the level set $\F^{-1}\{\F(\bfx^*)\}$ by
$$
\tilde{\U}_{\F}(\bfx):=\frac{\tilde{d}_{\F}^>(\bfx)\, a_{\F}^{\mathstrut}(\bfx)+\tilde{d}_{\F}^<(\bfx)\,
b_{\F}^{\mathstrut}(\bfx)}{\tilde{d}_{\F}^>(\bfx)+\tilde{d}_{\F}^<(\bfx)}
$$
(or by the corresponding limiting value as defined above), where
\begin{eqnarray*}
\tilde{d}_{\F}^<(\bfx) &:=& d_{\infty}(\bfx,[\inf\I,a_{\F}(\bfx^*)]^n)=\Max(\bfx)-a_{\F}(\bfx^*),\\
\tilde{d}_{\F}^>(\bfx) &:=& d_{\infty}(\bfx,[b_{\F}(\bfx^*),\sup\I]^n)=b_{\F}(\bfx^*)-\Min(\bfx).
\end{eqnarray*}
We then note that $\tilde{d}_{\F}^>(\bfx)+\tilde{d}_{\F}^<(\bfx)>0$ so that the new function $\M_{\F}$ is well defined on $\I^n$.

\begin{remark}\label{rem:s9d8f7}
\begin{enumerate}
\item[(i)] When any of the conditions (\ref{eq:Case1N}) and (\ref{eq:Case2N}) hold, the proposed modification of $\M_{\F}$ is necessary to
ensure the solvability condition $\M_{\F}(\bfx)\in\delta_{\F}^{-1}\{\F(\bfx)\}$. Indeed, if e.g.\
$a_{\F}(\bfx^*)\notin\delta_{\F}^{-1}\{\F(\bfx^*)\}$, then $\M_{\F}$ must satisfy $\M_{\F}(\bfx^*)>a_{\F}(\bfx^*)$, which fails to hold with the
original definition (\ref{eq:Urysohn}) of $\M_{\F}$ whenever $\bfx^*\in\Omega_{\F}$ and $d_{\F}^<(\bfx^*)=0$. For instance, consider $\F\colon
[0,1]^2\to\R$ defined by $\F\equiv 1$ on $[\frac 12,1]^2\setminus\{\frac 12,\frac 12\}$ and $\F\equiv 0$ elsewhere. Then, for $\bfx^*=(\frac
34,\frac 12)$, we have $\U_{\F}(\bfx^*)=\frac 12$ and hence $(\delta_{\F}\circ\U_{\F})(\bfx^*)=\delta_{\F}(\frac 12)=0\neq 1=\F(\bfx^*)$. To
solve this problem, we consider $\M_{\F}=\Max$ on $[\frac 12,1]^2$ and $\M_{\F}=\Min$ elsewhere, and then we have $\delta_{\F}\circ\M_{\F}=\F$.

\item[(ii)] It is immediate to see that none of the conditions (\ref{eq:Case1N}) and (\ref{eq:Case2N}) hold as soon as $\delta_{\F}$ is a
continuous function, in which case condition (\ref{eq:ranFrandF}) immediately follows.
\end{enumerate}
\end{remark}

The next theorem essentially states that, thus defined, the function $\M_{\F}\colon\I^n\to\R$ is a nondecreasing and idempotent solution to
Chisini's equation (\ref{eq:FuncEqFG2}).

\begin{theorem}\label{thm:GF-ND-Id}
For any nondecreasing function $\F\colon\I^n\to\R$ satisfying (\ref{eq:ranFrandF}), we have $\ran(\M_{\F})\subseteq\I$ and
$\F=\delta_{\F}\circ\M_{\F}$. Moreover, $\M_{\F}$ is nondecreasing and idempotent.
\end{theorem}

\begin{proof}
See Appendix~\ref{app:GF-ND-Id}.
\end{proof}

\begin{remark}\label{rem:GF-ND-Id}
\begin{enumerate}
\item[(i)] We will see in Section~\ref{sec:means} that nondecreasing and idempotent solutions of Chisini's equation are of particular interest. We
will call those solutions \emph{Chisini means} or \emph{level surface means} exactly as in the simple case when $\delta_{\F}$ is one-to-one.
Theorem~\ref{thm:GF-ND-Id} actually provides such a solution in a constructive way.

\item[(ii)] The idea of the construction of $\M_{\F}$ is the following. Let $\bfx^*\in\Omega_{\F}$ and, to keep the description simple, assume
that conditions (\ref{eq:Case1N}) and (\ref{eq:Case2N}) fail to hold. Then, on the whole level set $\F^{-1}\{\F(\bfx^*)\}$, we consider the
classical Urysohn function (hence the notation $\U_{\F}$) used in metric spaces, i.e., a continuous function defined by an inverse
distance-weighted average of the values $a_{\F}(\bfx^*)$ and $b_{\F}(\bfx^*)$:
\begin{equation}\label{eq:UryShe}
\U_{\F}(\bfx) =\frac{\frac{1}{d_{\F}^<(\bfx)}\, a_{\F}^{\mathstrut}(\bfx^*)+\frac{1}{d_{\F}^>(\bfx)}\,
b_{\F}^{\mathstrut}(\bfx^*)}{\frac{1}{d_{\F}^<(\bfx)}+\frac{1}{d_{\F}^>(\bfx)}}.
\end{equation}
Thus, the value $\U_{\F}(\bfx)$ partitions the interval $[a_{\F}(\bfx^*),b_{\F}(\bfx^*)]$ into two subintervals whose lengths are proportional
to $d_{\F}^<(\bfx)$ and $d_{\F}^>(\bfx)$, respectively. The two-dimensional case is illustrated in Figure~\ref{fig:GIU}. Moreover, looking into
the proof of Theorem~\ref{thm:GF-ND-Id}, it is easy to see that, from among all the Minkowski distances that could have been chosen to define
$\U_{\F}$, only the Chebyshev distance always ensures the nondecreasing monotonicity and idempotency of $\U_{\F}$.
\setlength{\unitlength}{.065\linewidth}
\begin{figure}[htb]
\begin{center}
\begin{picture}(8,9.2)
\put(0,1){\framebox(8,8)}%
\put(0,1){\line(1,1){8}}%
\put(3.5,9){\line(2,-1){4.5}}\put(0,3){\line(1,-1){2}}%
\put(0.5,2.5){\dashbox{.125}(3.5,3.5)}\put(4,6){\dashbox{.125}(1.83,1.83)}%
\put(4,6){\circle*{.13}}\put(3.75,6.25){\makebox(0,0){$\bfx$}}%
\put(1.3,8.62){\makebox(0,0){$\F^{-1}\{\F(\bfx^*)\}$}}\put(6.2,8.62){\makebox(0,0){$\F^{-1}_>(\F(\bfx^*))$}}%
\put(1,0.9){\line(0,1){0.2}}\put(1,0.5){\makebox(0,0){$a_{\F}(\bfx^*)$}}%
\put(6.5,0.9){\line(0,1){0.2}}\put(6.5,0.5){\makebox(0,0){$b_{\F}(\bfx^*)$}}%
\put(4.61,0.9){\line(0,1){0.2}}\put(4.61,0.5){\makebox(0,0){$\U_{\F}(\bfx)$}}%
\put(2.25,5.62){\makebox(0,0){$d_{\F}^<(\bfx)$}}\put(4.92,5.62){\makebox(0,0){$d_{\F}^>(\bfx)$}}%
\put(1,1){\dashbox{.125}(0,1)}\put(6.5,1){\dashbox{.125}(0,6.5)}
\end{picture}
\caption{Geometric interpretation of the function $\U_{\F}$} \label{fig:GIU}
\end{center}
\end{figure}
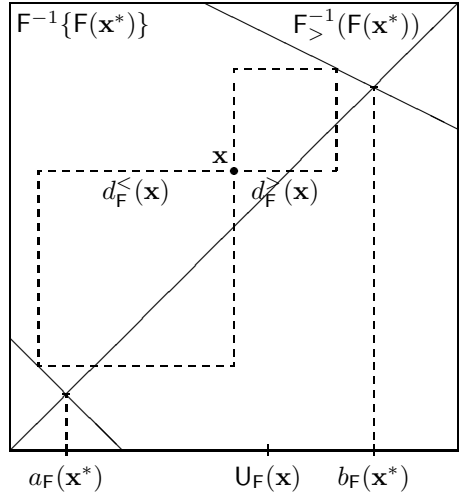

\item[(iii)] The definition of $\U_{\F}$, as given in (\ref{eq:UryShe}), recalls Shepard's metric interpolation technique \cite{GorWix78}, which
can be described as follows. Consider a function $F\colon\R^n\to\R$ and $p$ points $\bfx^{(1)},\ldots,\bfx^{(p)}\in\R^n$. Then, for any metric
$d$ on $\R^n$, the continuous extension of the function $U\colon\R^n\to\R$ defined by
$$
U(\bfx)=\sum_{k=1}^p\frac{F(\bfx^{(k)})}{d(\bfx,\bfx^{(k)})}\bigg{/}\sum_{k=1}^p\frac{1}{d(\bfx,\bfx^{(k)})}
$$
interpolates $F$ at the points $\bfx^{(1)},\ldots,\bfx^{(p)}$. By
letting $p=2$ and replacing the interpolating points by the lower and upper level sets of $\F$, we retrieve (\ref{eq:UryShe}) immediately.
\end{enumerate}
\end{remark}

\begin{example}\label{ex:FcontGnotcont}
Consider the continuous function $\F\colon [0,1]^2\to [0,1]$ defined by
$$
\textstyle{\F(x_1,x_2):=\Min\big(x_1,x_2,\frac 14+\Max(0,x_1+x_2-1)\big)}.
$$
Thus defined, $\F$ is an \emph{ordinal sum}\/ constructed from the {\L}ukasiewicz t-norm; see e.g.\ \cite{KleMes05}. Figure~\ref{fig:F1} shows
the 3D plot and the contour plot of $\F$. Figure~\ref{fig:G1} shows those of the function $\M_{\F}$. Note that the restriction of $\M_{\F}$ to
the open triangle of vertices $(\frac 14,\frac 34)$, $(\frac 14,\frac 14)$, $(\frac 34,\frac 14)$ is the function $\U_{\F}$, with
$a_{\F}(x_1,x_2)=\frac 14$, $b_{\F}(x_1,x_2)=\frac 12$, $d_{\F}^{<}(x_1,x_2)=\Min(x_1,x_2)-\frac 14$, and $d_{\F}^{>}(x_1,x_2)=\frac 12-\frac
12(x_1+x_2)$.
\end{example}

\begin{figure}[htb]
\includegraphics[height=.47\linewidth, keepaspectratio=true]{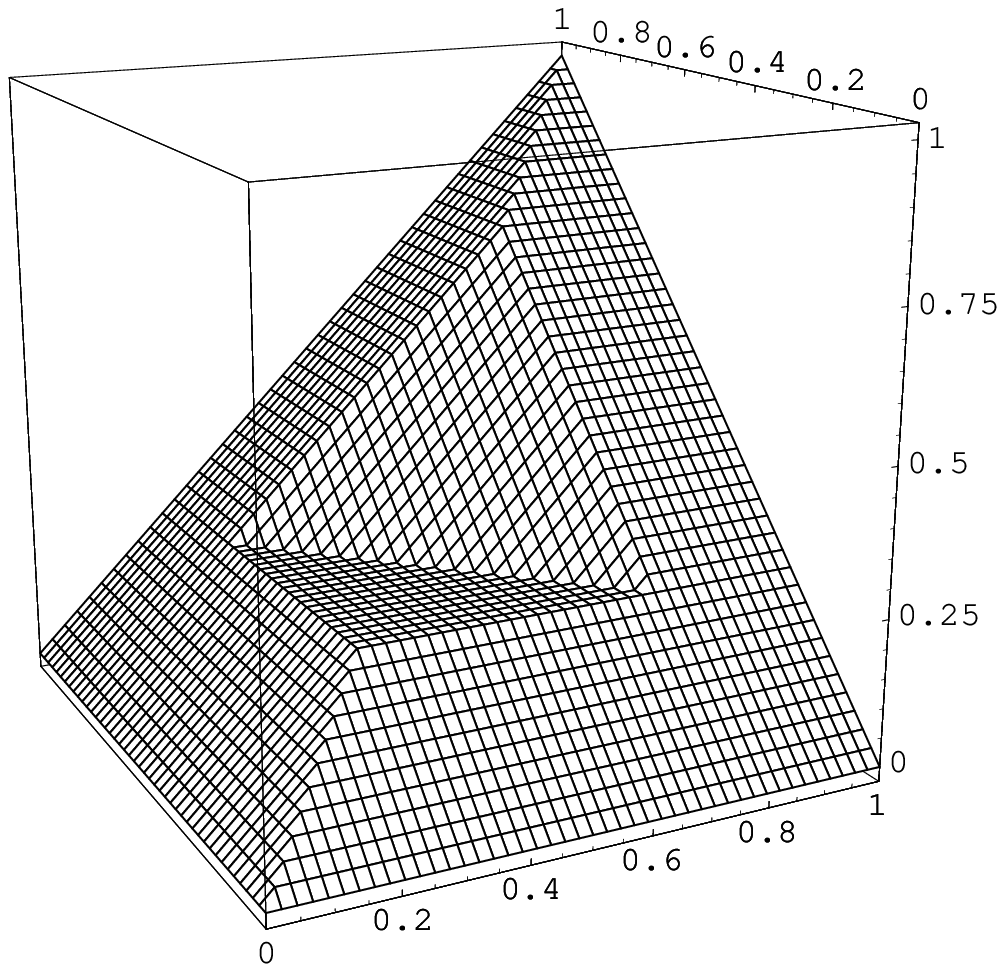}\hfill
\includegraphics[height=.45\linewidth, keepaspectratio=true]{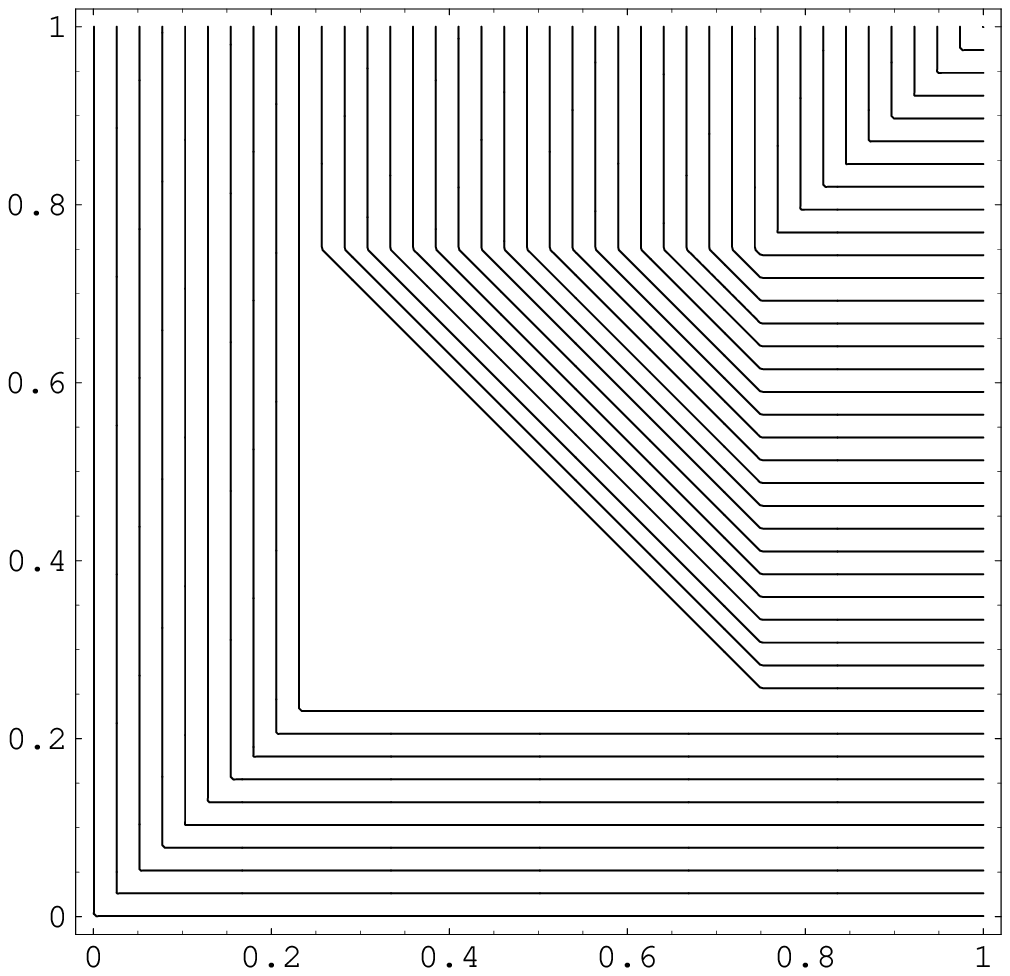}
\caption{Function $\F$ of Example~\ref{ex:FcontGnotcont} (3D plot and contour plot)} \label{fig:F1}
\end{figure}
\begin{figure}[htb]
\includegraphics[height=.47\linewidth, keepaspectratio=true]{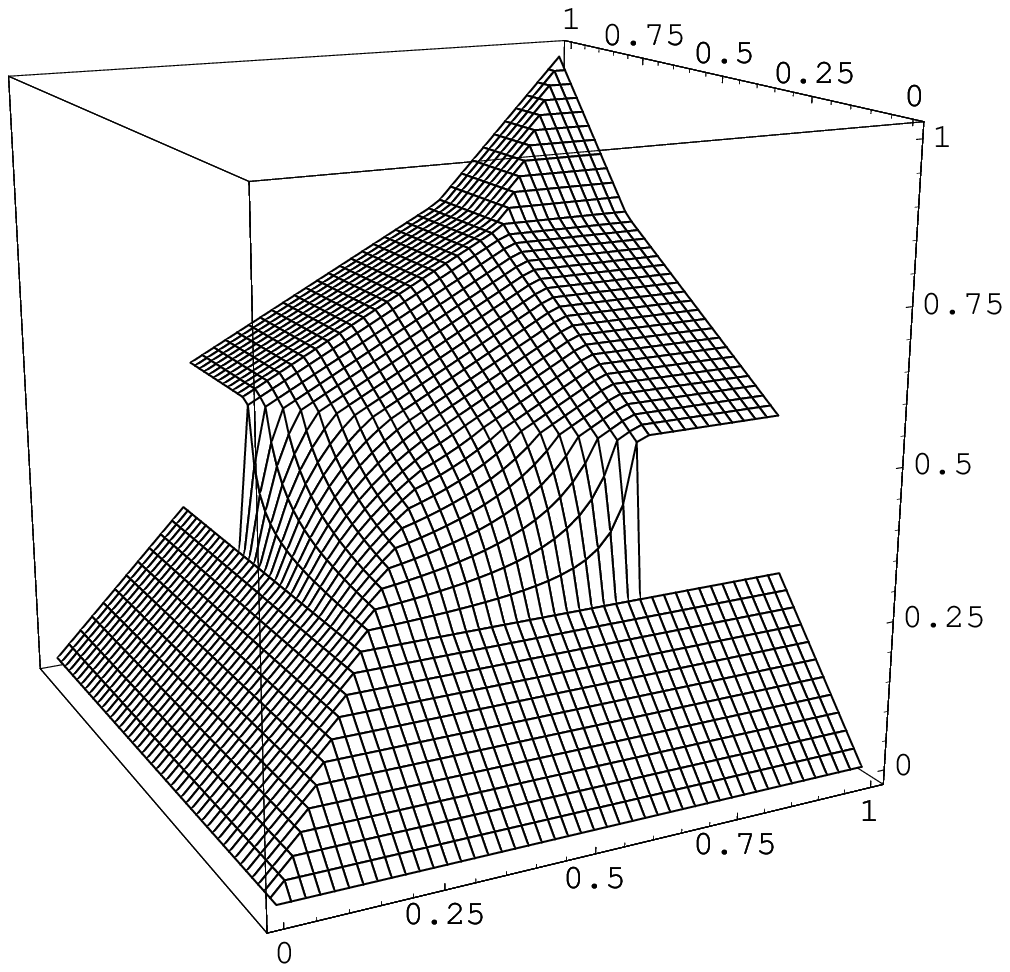}\hfill
\includegraphics[height=.45\linewidth, keepaspectratio=true]{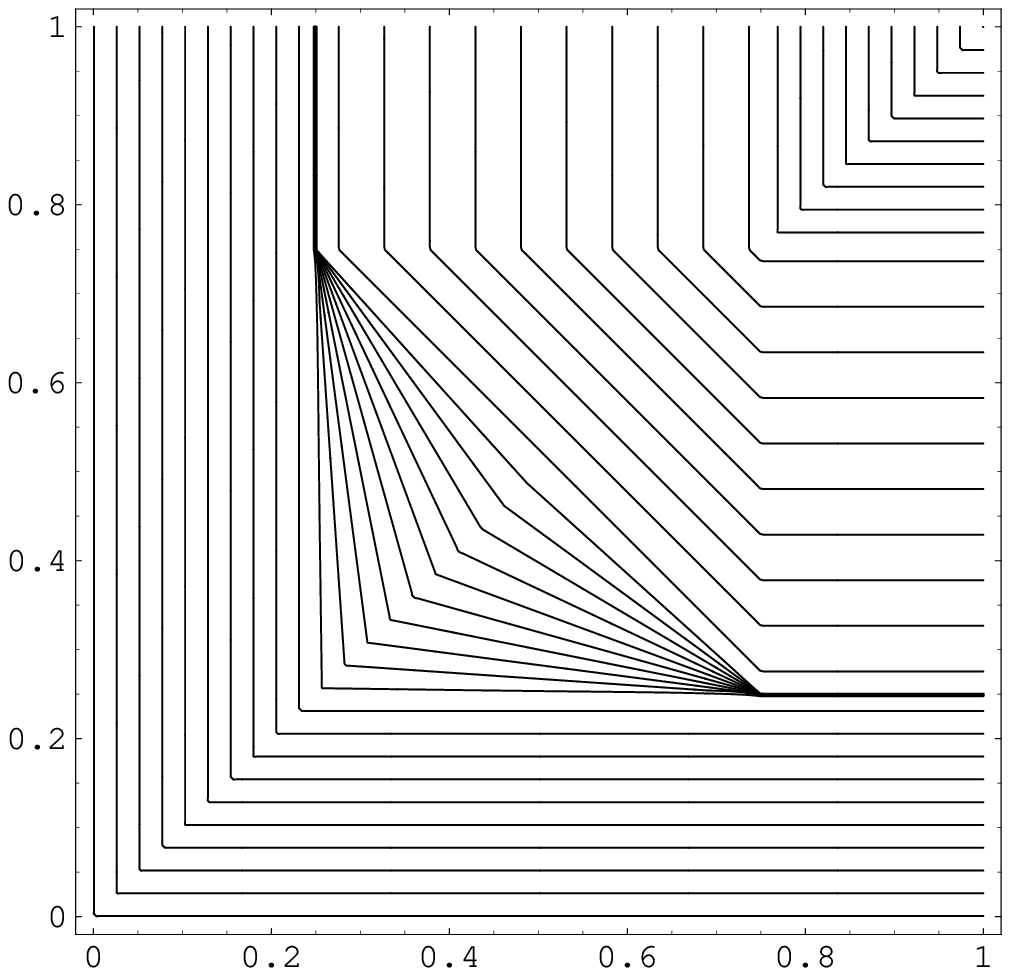}
\caption{Function $\M_{\F}$ of Example~\ref{ex:FcontGnotcont} (3D plot and contour plot)} \label{fig:G1}
\end{figure}

We now discuss a few properties of the solution $\M_{\F}$. Continuity issues will be discussed in the next section. We start with the following
straightforward result, which shows that $\M_{\F}$ can also be constructed from any strictly increasing transformation of $\F$.

\begin{proposition}\label{prop:TransfF2}
Let $\F\colon\I^n\to\R$ be a nondecreasing function satisfying condition (\ref{eq:ranFrandF}). For any strictly increasing function
$g\colon\ran(\delta_{\F})\to\R$, we have $\M_{\F}=\M_{g\circ\F}$.
\end{proposition}

Proposition~\ref{prop:TransfF2} has an important application. Any $g\in Q(\delta_{\F})$ such that $\dom(g)=\ran(\delta_\F)$ is strictly
increasing and $g\circ\F$ is range-idempotent (see Section~\ref{sec:TrivialS}). The calculation of $\M_{\F}$ can then be greatly simplified if we
consider $\M_{g\circ\F}$ instead. Observe for instance that, for every $\bfx\in\I^n$ such that $a_{\F}(\bfx)=b_{\F}(\bfx)$, we have
$\M_{g\circ\F}(\bfx)=\M_{\F}(\bfx)=(g\circ\F)(\bfx)$.

We now investigate the effect of \emph{dualization} of $\F$ over $\M_{\F}$ when $\F$ is defined on a compact domain $[a,b]^n$. Recall first the
concepts of \emph{dual} and \emph{self-dual} functions (see \cite{GarMar08} for a recent background). The \emph{dual} of a function $\F\colon
[a,b]^n\to [a,b]$ is the function $\F^d\colon [a,b]^n\to [a,b]$, defined by $\F^d=\psi\circ\F\circ(\psi,\ldots,\psi)$, where $\psi\colon
[a,b]\to [a,b]$ is the order-reversing involutive transformation $\psi(x)=a+b-x$ ($\psi^{-1}=\psi$). A function $\F\colon [a,b]^n\to [a,b]$ is
said to be \emph{self-dual} if $\F^d=\F$.

The following results essentially states that the map $\F\mapsto\M_{\F}$ commutes with dualization. In rough terms, our ``metric interpolation''
commutes with dualization.

\begin{proposition}\label{prop:dualization}
Let $\F\colon [a,b]^n\to [a,b]$ be a nondecreasing function satisfying condition (\ref{eq:ranFrandF}). Then $\M_{\F^d}=\M^d_{\F}$. In
particular, if $\F$ is self-dual then so is $\M_{\F}$.
\end{proposition}

\begin{proof}
It is straightforward to verify that $a_{\F^d}=\psi\circ b_{\F}\circ(\psi,\ldots,\psi)$, $b_{\F^d}=\psi\circ a_{\F}\circ(\psi,\ldots,\psi)$,
$d_{\F^d}^{<}=d_{\F}^{>}\circ(\psi,\ldots,\psi)$, $d_{\F^d}^{>}=d_{\F}^{<}\circ(\psi,\ldots,\psi)$,
$\tilde{d}_{\F^d}^{<}=\tilde{d}_{\F}^{>}\circ(\psi,\ldots,\psi)$, and $\tilde{d}_{\F^d}^{>}=\tilde{d}_{\F}^{<}\circ(\psi,\ldots,\psi)$. It is
then immediate to see that $\M_{\F^d}=\M^d_{\F}$.
\end{proof}

Although the map $\F\mapsto\M_{\F}$ commutes with dualization, it may not commute with restrictions, i.e., we may have
$\M_{\F|_{\J^n}}\neq\M_{\F}|_{\J^n}$ for some $\J\subseteq\I$. This shows that $\M_{\F}$ is not a ``local'' concept; its values depend not only
on $\F$ but also on the domain $\I^n$ considered. This fact can be illustrated by the binary function $\F(x_1,x_2)=\Min(x_1+x_2,\frac 12)$ over
the sets $\I=\R$ and $\J=[0,1]$. We have $\M_{\F}(x_1,x_2)=\frac 12(x_1+x_2)$ and
$$
\M_{\F|_{\J^2}}(x_1,x_2)=
\begin{cases}
\frac 12(x_1+x_2), & \mbox{if $|x_1-x_2|\leqslant\frac 12\, ,$}\\
\Max(x_1,x_2)-\frac 14\, , & \mbox{if $|x_1-x_2|\geqslant\frac 12\, .$}
\end{cases}
$$
However, the following result shows that, although $\M_{\F|_{\J^n}}$ and $\M_{\F}|_{\J^n}$ may be different, both functions solve the Chisini
equation associated with $\F|_{\J^n}$.

\begin{proposition}
Let $\F\colon\I^n\to\R$ be a nondecreasing function satisfying condition (\ref{eq:ranFrandF}) and let $\J\subseteq\I$. Then
$\F|_{\J^n}=\delta_{\F|_{\J^n}}\circ(\M_{\F}|_{\J^n})$.
\end{proposition}

\begin{proof}
We have $\F|_{\J^n}=(\delta_{\F}\circ\M_{\F})|_{\J^n}=\delta_{\F}\circ\M_{\F}|_{\J^n}$. Since $\M_{\F}|_{\J^n}$ is nondecreasing and idempotent,
it takes on its values in $\J$. Hence the result.
\end{proof}

As far as the symmetries of $\F$ are concerned, we also have the following result.

\begin{proposition}\label{prop:symm}
Let $\F\colon\I^n\to\R$ be a nondecreasing function satisfying condition (\ref{eq:ranFrandF}). Then $\sigma\in\mathfrak{S}_n$ is a symmetry of
$\F$ if and only if it is a symmetry of $\M_{\F}$.
\end{proposition}

\begin{proof}
The condition is clearly sufficient. Let us show that it is necessary. We can assume without loss of generality that $\sigma$ is a transposition
$(ij)$, with $i,j\in [n]$, $i\neq j$. Clearly, $\sigma$ is a symmetry of $a_{\F}$, $b_{\F}$, $\tilde{d}_{\F}^<$, and $\tilde{d}_{\F}^>$. To show
that it is also a symmetry of both $d_{\F}^<$ and $d_{\F}^>$, we only need to show that, for any given $\bfx\in\I^n$, the level set
$\F^{-1}\{\F(\bfx)\}$ is symmetric with respect to the hyperplane $x_i=x_j$. Let $\bfx'\in \F^{-1}\{\F(\bfx)\}$. Then $\F(\bfx')=\F(\bfx)$ and
hence $\F([\bfx']_{\sigma})=\F([\bfx]_{\sigma})$. That is, $[\bfx']_{\sigma}\in \F^{-1}\{\F([\bfx]_{\sigma})\}=\F^{-1}\{\F(\bfx)\}$.
\end{proof}

\section{Nondecreasing, idempotent, and continuous solutions}

In this section, assuming again that $\F$ is nondecreasing, we yield necessary and sufficient conditions on $\F$ for the associated Chisini
equation to have at least one continuous solution. We also show that the idempotent solution $\M_{\F}$ is continuous whenever a continuous
solution exists (see Theorem~\ref{thm:MainCont}). As we shall see, continuous solutions may exist even if $\F$ is not continuous. However, given
a continuous function $\F$, the associated Chisini equation may have no continuous solutions. Thus, surprisingly enough, continuity of $\F$ is
neither necessary nor sufficient to ensure the existence of continuous solutions.

The following lemma states that if a continuous solution of Chisini's equation exists then $\delta_{\F}^{-1}\{\F(\bfx)\}$ must be a singleton
for every $\bfx\in\I^n\setminus\Omega_{\F}$. Equivalently,
\begin{equation}\label{eq:Cond1C}
a_{\F}(\bfx)=b_{\F}(\bfx),\qquad \forall\bfx\in\I^n\setminus\Omega_{\F}.
\end{equation}

\begin{lemma}\label{lemma:disc}
Let $\F\colon\I^n\to\R$ be a nondecreasing function satisfying condition (\ref{eq:ranFrandF}) and let $\G\colon\I^n\to\I$ be any solution of
Chisini's equation (\ref{eq:FuncEqFG2}). Suppose there exists $\bfx^*\in\I^n\setminus\Omega_{\F}$ such that $a_{\F}(\bfx^*)< b_{\F}(\bfx^*)$.
Then $\bfx^*$ is a discontinuity point of $\G$.
\end{lemma}

\begin{proof}
Let $\G\colon\I^n\to\I$ be a solution of Chisini's equation (\ref{eq:FuncEqFG2}) and assume that there exists
$\bfx^*\in\I^n\setminus\Omega_{\F}$ such that $a_{\F}(\bfx^*)< b_{\F}(\bfx^*)$. It follows immediately that $\bfx^*\notin\mathrm{diag}(\I^n)$.
Now, since $d_{\F}^>(\bfx^*)=d_{\F}^<(\bfx^*)=0$, there exist unit vectors $\mathbf{u},\mathbf{v}\in\R^n$, with nonnegative components, and a
number $h^*>0$ such that $\F(\bfx^*-h\mathbf{u})<\F(\bfx^*)<\F(\bfx^*+h\mathbf{v})$ for all $h\in\left]0,h^*\right[$. Since
$\G(\bfx)\in\delta_{\F}^{-1}\{\F(\bfx)\}$ for all $\bfx\in\I^n$, it follows that $\G(\bfx^*-h\mathbf{u})\leqslant a_{\F}(\bfx^*)<
b_{\F}(\bfx^*)\leqslant \G(\bfx^*+h\mathbf{v})$ for all $h\in\left]0,h^*\right[$, which means that $\G$ is discontinuous at $\bfx^*$.
\end{proof}

The following example shows that, even if $\F$ is continuous, we may have $a_{\F}\neq b_{\F}$ on $\I^n\setminus\Omega_{\F}$, in which case the
corresponding Chisini equation has no continuous solutions.

\begin{example}\label{ex:FcontGnotcont2}
Consider again the function $\F$ described in Example~\ref{ex:FcontGnotcont}. We can easily see that any solution $\G\colon\I^2\to\I$ of the
Chisini equation (\ref{eq:FuncEqFG2}) is discontinuous along the line segment $\,]\frac 34,1]\times \{\frac 14\}$ (and, by symmetry, along the
line segment $\{\frac 14\}\times\,]\frac 34,1]$). Indeed, for any $x\in\,]\frac 34,1]$ and any $0<h<\frac 18$, we have $\F(x,\frac 14\pm
h)=\frac 14\pm h$, which implies that $(x,\frac 14)\in [0,1]^2\setminus\Omega_{\F}$. However, $\delta_{\F}^{-1}\{\frac 14+h\} = \{\frac
{1+h}2\}$, $\delta_{\F}^{-1}\{\frac 14-h\} = \{\frac 14-h\}$, and $\delta_{\F}^{-1}\{\frac 14\} = [\frac 14,\frac 12]$, which shows that no
function $\G\colon\I^2\to\I$ satisfying $\G(\bfx)\in\delta_{\F}^{-1}\{\F(\bfx)\}$ is continuous.
\end{example}

We now show that, if a continuous solution of Chisini's equation exists, then the following conditions must hold:
\begin{eqnarray}
d_{\F}^<(\bfx)=0 &\Rightarrow & a_{\F}(\bfx)\in\delta_{\F}^{-1}\{\F(\bfx)\},\qquad\forall\bfx\in\I^n,\label{eq:condNEW1}\\
d_{\F}^>(\bfx)=0 &\Rightarrow & b_{\F}(\bfx)\in\delta_{\F}^{-1}\{\F(\bfx)\},\qquad\forall\bfx\in\I^n.\label{eq:condNEW2}
\end{eqnarray}

\begin{lemma}\label{lemma:discNEW2}
Let $\F\colon\I^n\to\R$ be a nondecreasing function satisfying condition (\ref{eq:ranFrandF}) and let $\G\colon\I^n\to\I$ be any solution of
Chisini's equation (\ref{eq:FuncEqFG2}). Suppose that any of the conditions (\ref{eq:condNEW1}) and (\ref{eq:condNEW2}) are violated by some
$\bfx^*\in\I^n$. Then $\bfx^*$ is a discontinuity point of $\G$.
\end{lemma}

\begin{proof}
Let $\G\colon\I^n\to\I$ be a solution of Chisini's equation (\ref{eq:FuncEqFG2}) and suppose that (\ref{eq:condNEW1}) is violated by
$\bfx^*\in\I^n$. The other case can be dealt with dually. We clearly have $\bfx^*\notin\mathrm{diag}(\I^n)$. Now, since $d_{\F}^<(\bfx^*)=0$,
there exists a unit vector $\mathbf{u}\in\R^n$, with nonnegative components, and a number $h^*>0$ such that $\F(\bfx^*-h\mathbf{u})<\F(\bfx^*)$
for all $h\in\left]0,h^*\right[$. Since $\G(\bfx)\in\delta_{\F}^{-1}\{\F(\bfx)\}$ for all $\bfx\in\I^n$, we must have
$\G(\bfx^*-h\mathbf{u})\leqslant a_{\F}(\bfx^*)< \G(\bfx^*)$ for all $h\in\left]0,h^*\right[$. If $\G$ were continuous at $\bfx^*$, then we
would have $\G(\bfx^*)=a_{\F}(\bfx^*)$. But then $(\delta_{\F}\circ\G)(\bfx^*)=(\delta_{\F}\circ a_{\F})(\bfx^*)\neq\F(\bfx^*)$, a
contradiction.
\end{proof}

In the next lemma, we give two further necessary conditions for the existence of a continuous solution, namely
\begin{eqnarray}
\lim_{h\to 0^-} b_{\F}(\bfx+h\mathbf{e}_i) ~\geqslant ~ a_{\F}(\bfx),&&\quad\forall\bfx\in\I^n,\,\forall i\in [n], \label{eq:Cond2C}\\
\lim_{h\to 0^+} a_{\F}(\bfx+h\mathbf{e}_i) ~\leqslant ~ b_{\F}(\bfx),&&\quad\forall\bfx\in\I^n,\,\forall i\in [n].\label{eq:Cond3C}
\end{eqnarray}

\begin{lemma}\label{lemma:Gdisc3}
Let $\F\colon\I^n\to\R$ be a nondecreasing function satisfying condition (\ref{eq:ranFrandF}) and let $\G\colon\I^n\to\I$ be any solution of
Chisini's equation (\ref{eq:FuncEqFG2}). Let $\bfx^*\in\I^n$ and assume there are $i\in [n]$ and $h<0$ (resp.\ $h>0$) such that
$\bfx^*+h\bfe_i\in\I^n$. If $\lim_{h\to 0^-}b_{\F}(\bfx^*+h\bfe_i)<a_{\F}(\bfx^*)$ (resp.\ $\lim_{h\to
0^+}a_{\F}(\bfx^*+h\bfe_i)>b_{\F}(\bfx^*)$) then $\lim_{h\to 0^-}\G(\bfx^*+h\bfe_i)<\G(\bfx^*)$ (resp.\ $\lim_{h\to
0^+}\G(\bfx^*+h\bfe_i)>\G(\bfx^*)$).
\end{lemma}

\begin{proof}
We prove the statement related to the left-discontinuity of $\G$. The other one can be proved dually. Under the assumptions of the lemma, there
exist $h^*<0$ and $\varepsilon >0$ such that $b_{\F}(\bfx^*+h\bfe_i)\leqslant a_{\F}(\bfx^*)-\varepsilon$ for all $h\in\left]h^*,0\right[$. It
follows that $\G(\bfx^*+h\bfe_i)\leqslant b_{\F}(\bfx^*+h\bfe_i)\leqslant a_{\F}(\bfx^*)-\varepsilon\leqslant\G(\bfx^*)-\varepsilon$ for all
$h\in\left]h^*,0\right[$, which proves the result.
\end{proof}

The converse of Lemma~\ref{lemma:Gdisc3} does not hold in general. Indeed, when $\F$ is a constant function, any function $\G\colon\I^n\to\I$
(continuous or not) solves the corresponding Chisini equation. However, we now show that, assuming conditions (\ref{eq:Cond1C}),
(\ref{eq:condNEW1}), and (\ref{eq:condNEW2}), the converse of Lemma~\ref{lemma:Gdisc3} holds for the special solution $\M_{\F}$.

\begin{lemma}\label{lemma:Gdisc4}
Let $\F\colon\I^n\to\R$ be a nondecreasing function satisfying conditions (\ref{eq:ranFrandF}), (\ref{eq:Cond1C}), (\ref{eq:condNEW1}), and
(\ref{eq:condNEW2}). Let $\bfx^*\in\I^n$ and assume there are $i\in [n]$ and $h<0$ (resp.\ $h>0$) such that $\bfx^*+h\bfe_i\in\I^n$. We have
$\lim_{h\to 0^-}b_{\F}(\bfx^*+h\bfe_i)<a_{\F}(\bfx^*)$ (resp.\ $\lim_{h\to 0^+}a_{\F}(\bfx^*+h\bfe_i)>b_{\F}(\bfx^*)$) if and only if
$\lim_{h\to 0^-}\M_{\F}(\bfx^*+h\bfe_i)<\M_{\F}(\bfx^*)$ (resp.\ $\lim_{h\to 0^+}\M_{\F}(\bfx^*+h\bfe_i)>\M_{\F}(\bfx^*)$).
\end{lemma}

\begin{proof}
Again, we prove the left-discontinuity result. The other one can be proved dually. The necessity immediately follows from
Lemma~\ref{lemma:Gdisc3}. Let us prove the sufficiency. For the sake of a contradiction, suppose that
\begin{equation}\label{eq:Gdisc2}
\lim_{h\to 0^-}\M_{\F}(\bfx^*+h\bfe_i)<\M_{\F}(\bfx^*)\quad\mbox{and}\quad\lim_{h\to 0^-}b_{\F}(\bfx^*+h\bfe_i)\geqslant a_{\F}(\bfx^*).
\end{equation}
Due to (\ref{eq:condNEW1}) and (\ref{eq:condNEW2}), both conditions (\ref{eq:Case1N}) and (\ref{eq:Case2N}) fail to hold and hence $\M_{\F}$ is
given by (\ref{eq:Urysohn}). Two exclusive cases are to be examined:
\begin{enumerate}
\item[(i)] If $\bfx^*\in\Omega_{\F}$ then $\M_{\F}(\bfx^*)=\U_{\F}(\bfx^*)$.
\begin{enumerate}
\item[(a)] Suppose that there exists $h^*<0$ such that $\bfx^*+h\mathbf{e}_i\in\Omega_{\F}\cap\F^{-1}\{\F(\bfx^*)\}$ for all
$h\in\left]h^*,0\right[$. Then $\M_{\F}=\U_{\F}$ on the half-closed line segment $\left]\bfx^*+h^*\mathbf{e}_i,\bfx^*\right]$. This contradicts
(\ref{eq:Gdisc2}) since $\U_{\F}$ is continuous on each level set $\Omega_{\F}\cap\F^{-1}\{y\}$, with $y\in\ran(\delta_{\F})$.

\item[(b)] Suppose that there exists $h^*<0$ such that $\bfx^*+h\mathbf{e}_i\in\Omega_{\F}\cap\F_{<}^{-1}(\F(\bfx^*))$ for all
$h\in\left]h^*,0\right[$. Then there exists $h'\in\left]h^*,0\right[$ such that $\F$ is constant on the open line segment
$\left]\bfx^*+h'\mathbf{e}_i,\bfx^*\right[$ (otherwise $\bfx^*+h\mathbf{e}_i\notin\Omega_{\F}$). Therefore, $\lim_{h\to
0^-}\F(\bfx^*+h\mathbf{e}_i)<\F(\bfx^*)$ and hence $\lim_{h\to 0^-}d^>_{\F}(\bfx^*+h\mathbf{e}_i)=0$. This implies $\lim_{h\to
0^-}\M_{\F}(\bfx^*+h\bfe_i) = \lim_{h\to 0^-}b_{\F}(\bfx^*+h\bfe_i)$. However, we also have $d_{\F}^<(\bfx^*)=0$ and hence
$\M_{\F}(\bfx^*)=a_{\F}(\bfx^*)$, thus contradicting (\ref{eq:Gdisc2}).

\item[(c)] Suppose that there exists $h^*<0$ such that $\bfx^*+h\mathbf{e}_i\in \I^n\setminus\Omega_{\F}$ for all $h\in\left]h^*,0\right[$. Then
$\M_{\F}(\bfx^*+h\bfe_i)=b_{\F}(\bfx^*+h\bfe_i)$ for all $h\in\left]h^*,0\right[$ and we conclude as in case (b) above.

\end{enumerate}
\item[(ii)] If $\bfx^*\in\I^n\setminus\Omega_{\F}$ then $\M_{\F}(\bfx^*)=a_{\F}(\bfx^*)=b_{\F}(\bfx^*)$ (cf.\ condition (\ref{eq:Cond1C})).
\begin{enumerate}
\item[(a)] Suppose that there exists $h^*<0$ such that $\bfx^*+h\mathbf{e}_i\in\Omega_{\F}$ for all $h\in\left]h^*,0\right[$. Then there exists
$h'\in\left]h^*,0\right[$ such that $\F$ is constant on the line segment $\left]\bfx^*+h'\mathbf{e}_i,\bfx^*\right[$. It follows that
$\lim_{h\to 0^-}d^>_{\F}(\bfx^*+h\mathbf{e}_i)=0$ and we conclude as in case (b) above.

\item[(b)] Suppose that there exists $h^*<0$ such that $\bfx^*+h\mathbf{e}_i\in\I^n\setminus\Omega_{\F}$ for all $h\in\left]h^*,0\right[$. Then
$\M_{\F}=a_{\F}=b_{\F}$ on the half-closed line segment $\left]\bfx^*+h^*\mathbf{e}_i,\bfx^*\right]$ and this contradicts
(\ref{eq:Gdisc2}).\qedhere
\end{enumerate}
\end{enumerate}
\end{proof}

We now state our main result related to the existence of continuous solutions. We first recall an important result on nondecreasing functions.
For a detailed proof, see e.g.\ \cite[Chapter 2]{GraMarMesPap09}.

\begin{proposition}\label{prop:ND-ContC}
A nondecreasing function of $n$ variables is continuous if and only if it is continuous in each of its variables.
\end{proposition}

\begin{theorem}\label{thm:MainCont}
Let $\F\colon\I^n\to\R$ be a nondecreasing function satisfying condition (\ref{eq:ranFrandF}). Then the following assertions are equivalent:
\begin{enumerate}
\item[(i)] There exists a continuous solution of Chisini's equation (\ref{eq:FuncEqFG2}).

\item[(ii)] $\M_{\F}$ is a continuous solution of Chisini's equation (\ref{eq:FuncEqFG2}).

\item[(iii)] $\F$ satisfies conditions (\ref{eq:Cond1C}), (\ref{eq:condNEW1}), (\ref{eq:condNEW2}), (\ref{eq:Cond2C}), and (\ref{eq:Cond3C}).
\end{enumerate}
\end{theorem}

\begin{proof}
The implication $(ii)\Rightarrow (i)$ is immediate. The implication $(i)\Rightarrow (iii)$ follows from Lemmas~\ref{lemma:disc},
\ref{lemma:discNEW2}, and \ref{lemma:Gdisc3}. To complete the proof, it remains to show that $(iii)\Rightarrow (ii)$. By
Theorem~\ref{thm:GF-ND-Id}, $\M_{\F}$ is a nondecreasing solution of Chisini's equation. Hence, by Proposition~\ref{prop:ND-ContC}, it suffices
to show that $\M_{\F}$ is continuous in each variable, which follows immediately from Lemma~\ref{lemma:Gdisc4}.
\end{proof}

\begin{remark}\label{rem:Cont-W1E}
\begin{enumerate}
\item[(i)] Theorem~\ref{thm:MainCont} provides necessary and sufficient conditions on a nondecreasing function $\F\colon\I^n\to\R$ satisfying
condition (\ref{eq:ranFrandF}) for its associated Chisini equation to have continuous solutions. When these conditions are satisfied, then the
function $\M_{\F}$ is a nondecreasing, idempotent, and continuous solution.

\item[(ii)] The following examples show that the conditions mentioned in assertion (iii) of Theorem~\ref{thm:MainCont} are independent:
\begin{enumerate}
\item[(a)] The function $\F$ in Example~\ref{ex:FcontGnotcont} satisfies all but condition (\ref{eq:Cond1C}).

\item[(b)] Consider an idempotent and noncontinuous function $\F$. Then conditions (\ref{eq:Cond1C}), (\ref{eq:condNEW1}), and
(\ref{eq:condNEW2}) are clearly satisfied but $\M_{\F}=\F$ is noncontinuous, which shows that (\ref{eq:Cond2C}) or (\ref{eq:Cond3C}) fails.

\item[(c)] The example given in Remark~\ref{rem:s9d8f7} (i) satisfies all but condition (\ref{eq:condNEW1}) and a dual example would make
(\ref{eq:condNEW2}) fail.
\end{enumerate}
\end{enumerate}
\end{remark}

The following two corollaries particularize Theorem~\ref{thm:MainCont} to the cases when $\delta_{\F}$ is continuous and when $\F$ is
continuous. As already observed in Example~\ref{ex:FcontGnotcont2}, continuity of $\F$ does not ensure the existence of continuous solutions.
These corollaries show that condition (\ref{eq:Cond1C}) remains the key property of $\F$ to ensure the existence of continuous solutions.

We first consider a lemma.

\begin{lemma}\label{lemma:NegDisc}
Let $\F\colon\I^n\to\R$ be a nondecreasing and continuous function. Then the function $a_{\F}$ (resp.\ $b_{\F}$) is left-continuous (resp.\
right-continuous) in each variable.
\end{lemma}

\begin{proof}
Let us establish the result for $a_{\F}$ only. The other function can be dealt with similarly. Let $i\in [n]$ and, for the sake of
contradiction, suppose that there exist $h^*<0$, $\varepsilon >0$, and $\bfx\in\I^n$ such that $a_{\F}(\bfx^*+h\mathbf{e}_i)\leqslant
a_{\F}(\bfx^*)-\varepsilon$ for all $h\in\left]h^*,0\right[$. By nondecreasing monotonicity of $\delta_{\F}$,
$$
\F(\bfx^*+h\mathbf{e}_i)=\delta_{\F}(a_{\F}(\bfx^*+h\mathbf{e}_i))\leqslant\delta_{\F}(a_{\F}(\bfx^*)-\varepsilon)\leqslant\delta_{\F}(a_{\F}(\bfx^*))=\F(\bfx^*)
$$
for all $h\in\left]h^*,0\right[$. By continuity of $\F$, we must have $\F(\bfx^*)=\delta_{\F}(a_{\F}(\bfx^*)-\varepsilon)$ and hence
$$a_{\F}(\bfx^*)=\inf\delta_{\F}^{-1}\{\F(\bfx^*)\}=\inf\delta_{\F}^{-1}\{\delta_{\F}(a_{\F}(\bfx^*)-\varepsilon)\}\leqslant
a_{\F}(\bfx^*)-\varepsilon,
$$
a contradiction.
\end{proof}

\begin{corollary}\label{cor:ContDisc2}
Let $\F\colon\I^n\to\R$ be a nondecreasing function such that $\delta_{\F}$ is continuous. Then the following assertions are equivalent:
\begin{enumerate}
\item[(i)] There exists a continuous solution of Chisini's equation (\ref{eq:FuncEqFG2}).

\item[(ii)] $\M_{\F}$ is a continuous solution of Chisini's equation (\ref{eq:FuncEqFG2}).

\item[(iii)] $\F$ satisfies conditions (\ref{eq:Cond1C}), (\ref{eq:Cond2C}), and (\ref{eq:Cond3C}).
\end{enumerate}
\end{corollary}

\begin{proof}
Since $\delta_{\F}$ is continuous, the function $\F$ satisfies conditions (\ref{eq:ranFrandF}), (\ref{eq:condNEW1}), and (\ref{eq:condNEW2});
see Remark~\ref{rem:s9d8f7} (ii). We then conclude by Theorem~\ref{thm:MainCont}.
\end{proof}

\begin{corollary}\label{cor:ContDisc3}
Let $\F\colon\I^n\to\R$ be a nondecreasing and continuous function. Then the following assertions are equivalent:
\begin{enumerate}
\item[(i)] There exists a continuous solution of Chisini's equation (\ref{eq:FuncEqFG2}).

\item[(ii)] $\M_{\F}$ is a continuous solution of Chisini's equation (\ref{eq:FuncEqFG2}).

\item[(iii)] $\F$ satisfies condition (\ref{eq:Cond1C}).
\end{enumerate}
\end{corollary}

\begin{proof}
By Lemma~\ref{lemma:NegDisc}, $\F$ satisfies conditions (\ref{eq:Cond2C}) and (\ref{eq:Cond3C}). We then conclude by
Corollary~\ref{cor:ContDisc2}.
\end{proof}

\section{Applications}

We briefly describe four applications for these special solutions of Chisini's equation: revisiting the concept of Chisini mean, proposing and
investigating generalizations of idempotency, extending the idempotization process to nondecreasing functions whose diagonal section is not
one-to-one, and characterizing certain transformed continuous functions.

\subsection{The concepts of mean and average revisited}
\label{sec:means}

The study of Chisini's functional equation enables us to better understand the concepts of mean and average. Already discovered and studied by
the ancient Greeks (see e.g.\ \cite[Chapter~3]{Ant98}), the concept of mean has given rise today to a very wide field of investigation with a
huge variety of applications. For general background, see \cite{Bul03,GraMarMesPap09}.

The first modern definition of mean was probably due to Cauchy~\cite{Cau21} who considered in 1821 a mean as an \emph{internal} function, i.e.,
a function $\M\colon\I^n\to\I$ satisfying $\Min\leqslant\M\leqslant\Max$. As it is natural to ask a mean to be nondecreasing, we say that a
function $\M\colon\I^n\to\I$ is a \emph{mean} in $\I^n$ if it is nondecreasing and internal. As a consequence, every mean is idempotent.
Conversely, any nondecreasing and idempotent function is internal and hence is a mean. This well-known fact follows from the immediate
inequalities
$$
\delta_{\M}\circ\Min\leqslant\M\leqslant\delta_{\M}\circ\Max.
$$
Moreover, if $\M\colon\I^n\to\I$ is a mean in $\I^n$ then, for any
subinterval $\J\subseteq\I$, $\M$ is also a mean in $\J^n$.

The concept of mean as an average is usually ascribed to Chisini~\cite[p.~108]{Chi29}, who defined in 1929 a mean associated with a function
$\F\colon\I^n\to\R$ as a solution $\M\colon\I^n\to\I$ of the equation $\F=\delta_{\F}\circ\M$. Unfortunately, as noted by de
Finetti~\cite[p.~378]{deF31} in 1931, Chisini's definition is so general that it does not even imply that the ``mean'' (provided there exists a
unique solution to Chisini's equation) satisfies the internality property. To ensure existence, uniqueness, nondecreasing monotonicity, and
internality of the solution of Chisini's equation it is enough to assume that $\F$ is nondecreasing and that $\delta_{\F}$ is a bijection from
$\I$ onto $\ran(\F)$ (see Corollary~\ref{cor:uniqueness}). Thus, we say that a function $\M\colon\I^n\to\I$ is an \emph{average}\/ in $\I^n$ if
there exists a nondecreasing function $\F\colon\I^n\to\R$, whose diagonal section $\delta_{\F}$ is a bijection from $\I$ onto $\ran(\F)$, such
that $\F=\delta_{\F}\circ\M$. In this case, we say that $\M=\delta_{\F}^{-1}\circ\F$ is the \emph{average associated with $\F$} (or the
\emph{$\F$-level mean} \cite[VI.4.1]{Bul03}) in $\I^n$.

Thus defined, the concepts of mean and average coincide. Indeed, any average is nondecreasing and idempotent and hence is a mean. Conversely,
any mean is the average associated with itself.

Now, by relaxing the strict increasing monotonicity of $\delta_{\F}$ into condition (\ref{eq:ranFrandF}), the existence (but not the uniqueness)
of solutions of the Chisini equation is still ensured (see Proposition~\ref{prop:FdFGrr}) and we have even seen that, if $\F$ is nondecreasing,
there are always means among the solutions (see Theorem~\ref{thm:GF-ND-Id}). This motivates the following general definition.

\begin{definition}\label{de:average2}
A function $\M\colon\I^n\to\I$ is an \emph{average} (or a \emph{Chisini mean}\/ or a \emph{level surface mean}) in $\I^n$ if it is a
nondecreasing and idempotent solution of the equation $\F=\delta_{\F}\circ\M$ for some nondecreasing function $\F\colon\I^n\to\R$. In this case,
we say that $\M$ is an \emph{average associated with $\F$} (or an \emph{$\F$-level mean}) in $\I^n$.
\end{definition}

Given a nondecreasing function $\F\colon\I^n\to\R$ satisfying (\ref{eq:ranFrandF}), the solution $\M_{\F}$ of the associated Chisini's equation
is a noteworthy $\F$-level mean. Indeed, it is a mean (see Theorem~\ref{thm:GF-ND-Id}) which has the same symmetries as $\F$ (see
Proposition~\ref{prop:symm}). Also, if $\F$ is continuous then $\M_{\F}$ is continuous if and only if $a_{\F}=b_{\F}$ on
$\I^n\setminus\Omega_{\F}$ (see Corollary~\ref{cor:ContDisc3}). Moreover, when $\I$ is compact, the map $\F\mapsto\M_{\F}$ commutes with
dualization (see Proposition~\ref{prop:dualization}).

\subsection{Quasi-idempotency and range-idempotency}
\label{sec:Quasi-Id}

Let $\F\colon\I^n\to\R$ be a function satisfying (\ref{eq:ranFrandF}). We have seen in Section~\ref{sec:TrivialS} that, assuming AC (not necessary
if $\delta_{\F}$ is monotonic), there exists an idempotent function $\G\colon\I^n\to\I$ such that $\F=\delta_{\F}\circ\G$. This result motivates
the following definition. We say that a function $\F\colon\I^n\to\R$ satisfying condition (\ref{eq:ranFrandF}) is \emph{quasi-idempotent}\/ if
$\delta_{\F}$ is monotonic. We say that it is \emph{idempotizable}\/ if $\delta_{\F}$ is strictly monotonic.

\begin{proposition}\label{prop:Quasi-Id}
Let $\F\colon\I^n\to\R$ be a function. Then the following assertions are equivalent:
\begin{enumerate}
\item[(i)] $\F$ is quasi-idempotent.

\item[(ii)] $\delta_{\F}$ is monotonic and there is a function $\G\colon\I^n\to\I$ such that $\F=\delta_{\F}\circ\G$.

\item[(iii)] $\delta_{\F}$ is monotonic and there is an idempotent function $\G\colon\I^n\to\I$ such that $\F=\delta_{\F}\circ\G$.

\item[(iv)] $\delta_{\F}$ is monotonic and there are functions $\G\colon\I^n\to\I$ and $f\colon\ran(\G)\to\R$ such that
$\ran(\delta_{\G})=\ran(\G)$ and $\F=f\circ\G$.

\item[(v)] $\delta_{\F}$ is monotonic and there are functions $\G\colon\I^n\to\I$ and $f\colon\ran(\G)\to\R$ such that $\G$ is idempotent and
$\F=f\circ\G$. In this case, $f=\delta_{\F}$.
\end{enumerate}
\end{proposition}

\begin{proof}
The solvability of Chisini's equation does not require AC since $\delta_{\F}$ is monotonic. This shows that $(i)\Leftrightarrow
(ii)\Leftrightarrow (iii)$. To prove that $(iii)\Rightarrow (v)$, just define $f:=\delta_{\F}|_{\ran(\G)}$ and observe that
$\F=\delta_{\F}\circ\G=f\circ\G$. Evidently, $(v)\Rightarrow (iv)$. Finally, to prove that $(iv)\Rightarrow (i)$, just observe that
$\ran(\delta_{\F})=\ran(f\circ\delta_{\G})=\ran(f\circ\G)=\ran(\F)$.
\end{proof}

\begin{corollary}\label{cor:idempotizable}
A function $\F\colon\I^n\to\R$ is idempotizable if and only if $\delta_{\F}$ is a strictly monotonic bijection from $\I$ onto $\ran(\F)$ and
there is a unique idempotent function $\G\colon\I^n\to\I$, namely $\G=\delta_{\F}^{-1}\circ\F$, such that $\F=\delta_{\F}\circ\G$.
\end{corollary}

Recall that a function $\F\colon\I^n\to\I$ is \emph{range-idempotent} if $\delta_{\F}\circ\F=\F$ (see Section~\ref{sec:TrivialS}). In this case,
$f:=\delta_{\F}$ necessarily satisfies the functional equation $f\circ f=f$, called the \emph{idempotency equation} \cite[Sect.~11.9E]{KucChoGer90}.
We can easily see \cite{Kuc61} that a function $f\colon\I\to\R$ solves this equation if and only if $f|_{\ran(f)}={\rm id}|_{\ran(f)}$; see also
\cite[Sect.~2.1]{SchSkl83}. The next two results characterize the family of nondecreasing solutions and the subfamily of nondecreasing and
continuous solutions of the idempotency equation. The proofs are straightforward and hence omitted.

\begin{proposition}\label{prop:Feigenbaum}
A nondecreasing function $f\colon\I\to\R$ satisfies $f\circ f=f$ if and only if the following conditions hold:
\begin{enumerate}
\item[(i)] If $f$ is strictly increasing on $\J\subseteq\I$ ($\J$ not a singleton) then $f|_{\J}=\id|_{\J}$.

\item[(ii)] If $f=c_{\J}$ is constant on $\J\subseteq\I$ then $c_{\J}\in f^{-1}\{c_{\J}\}$.
\end{enumerate}
\end{proposition}


\begin{corollary}\label{cor:Feigenbaum}
A nondecreasing and continuous function $f\colon\I\to\R$ satisfies $f\circ f=f$ if and only if there are $a,b\in\I\cup\{-\infty,\infty\}$,
$a\leqslant b$, with $a<b$ if $a\notin\I$ or $b\notin\I$, such that $f(x)=\Max(a,\Min(x,b))$.
\end{corollary}

\begin{remark}
Corollary~\ref{cor:Feigenbaum} was established in \cite[Sect.~2.2]{SchSkl83} when $\I$ is a bounded closed interval. It was also established in a
more general setting when the domain of variables is a bounded distributive lattice; see \cite{CouMar}.
\end{remark}

It is an immediate fact that a range-idempotent function $\F\colon\I^n\to\I$ with monotonic $\delta_{\F}$ is quasi-idempotent. Therefore, by
combining Proposition~\ref{prop:Quasi-Id} and Corollary~\ref{cor:Feigenbaum}, we see that a function $\F\colon\I^n\to\I$ is range-idempotent
with nondecreasing and continuous $\delta_{\F}$ if and only if there are $a,b\in\I\cup\{-\infty,\infty\}$, $a\leqslant b$, with $a<b$ if
$a\notin\I$ or $b\notin\I$, and an idempotent function $\G\colon\I^n\to\I$ such that
$$
\F(\bfx)=\Max(a,\Min(\G(\bfx),b)).
$$

\subsection{Idempotization process}

Corollary~\ref{cor:idempotizable} makes it possible to define an idempotent function $\G$ from any idempotizable function $\F$ (see Section~\ref{sec:Quasi-Id}), simply by writing $\G=\delta_{\F}^{-1}\circ\F$, hence the name ``idempotizable''. This generation process is known as
the \emph{idempotization process}; see \cite[Sect.~3.1]{CalKolKomMes02}. Of course, if $\F$ is nondecreasing then so is $\G$ and hence $\G$ is a
mean, namely the $\F$-level mean $\M_{\F}$ (see Section~\ref{sec:means}).

\begin{example}
From the \emph{Einstein sum}, defined on
$\left]-1,1\right[^2$ by
$$
\F(x_1,x_2)=\varphi^{-1}(\varphi(x_1)+\varphi(x_2))=\frac{x_1+x_2}{1+x_1x_2}\, ,
$$
where $\varphi=\mathrm{arctanh}$, we generate the quasi-arithmetic mean
$$
\M_{\F}(x_1,x_2)=\textstyle{\varphi^{-1}\big(\frac 12\varphi(x_1)+\frac 12\varphi(x_2)\big)}=\frac{1+x_1x_2-(1-x_1^2)^{1/2}(1-x_2^2)^{1/2}}{x_1+x_2}\, .
$$
\end{example}

Theorem~\ref{thm:GF-ND-Id} shows that we can extend this process to any nondecreasing and quasi-idempotent function $\F$ simply by considering
any $\F$-level mean (e.g., $\M_{\F}$). We call this process the \emph{generalized idempotization process}.

It may happen that $\M_{\F}$ be very difficult to calculate. The following result may then be helpful in obtaining alternative $\F$-level means.

\begin{proposition}\label{prop:TransfF}
Let $\F\colon\I^n\to\R$ be a nondecreasing function satisfying condition (\ref{eq:ranFrandF}), let $\J$ be a real interval, and let
$\F'\colon\J^n\to\R$ be defined by $\F':=\F\circ(\varphi,\ldots,\varphi)$, where $\varphi\colon\J\to\I$ is a strictly monotonic and continuous
function. Then, for any $\psi\in Q(\varphi)$, the function $\G'\colon\J^n\to\J$, defined by
$\G':=\psi\circ\M_{\F}\circ(\varphi,\ldots,\varphi)$,
\begin{itemize}
\item[(i)] is a well-defined $\F'$-level mean,

\item[(ii)] has the same symmetries as $\F$ and $\F'$, and

\item[(iii)] is continuous if $\F$ satisfies conditions (\ref{eq:Cond1C}), (\ref{eq:condNEW1}), (\ref{eq:condNEW2}), (\ref{eq:Cond2C}), and
(\ref{eq:Cond3C}).
\end{itemize}
\end{proposition}

\begin{proof}
Since $\M_{\F}$ is nondecreasing and idempotent, it is internal (see Section~\ref{sec:means}). Thus, $\varphi$ and
$\M_{\F}\circ(\varphi,\ldots,\varphi)$ have the same range and hence $\G'$ is well defined and even nondecreasing. Also, since $\varphi\in
Q(\psi)$ and $\ran(\psi)=\J$, we have $\delta_{\G'}=\psi\circ\varphi=\mathrm{id}$, which means that $\G'$ is idempotent. Moreover, we have
$$
\delta_{\F'}\circ\G' = \delta_{\F}\circ\varphi\circ\psi\circ\M_{\F}\circ(\varphi,\ldots,\varphi)=
\delta_{\F}\circ\M_{\F}\circ(\varphi,\ldots,\varphi)= \F\circ(\varphi,\ldots,\varphi)= \F',
$$
which shows that $\G'$ is an $\F'$-level mean.
Evidently, $\G'$ has the same symmetries as $\M_{\F}$ which, in turn, has the same symmetries as $\F$ (see Proposition~\ref{prop:symm}).
Finally, if $\F$ satisfies conditions (\ref{eq:Cond1C}), (\ref{eq:condNEW1}), (\ref{eq:condNEW2}), (\ref{eq:Cond2C}), and (\ref{eq:Cond3C}),
then $\M_{\F}$ is continuous (see Theorem~\ref{thm:MainCont}) and, since both $\varphi$ and $\psi$ are continuous, so is $\G'$.
\end{proof}

\begin{example}
The \emph{continuous Archimedean t-norm} $\mathsf{T}^{\varphi}\colon [0,1]^2\to [0,1]$ generated by the continuous strictly decreasing function
$\varphi\colon [0,1]\to [0,\infty]$, with $\varphi(1)=0$, is defined by
$$
\mathsf{T}^{\varphi}(x_1,x_2)=\psi(\varphi(x_1)+\varphi(x_2)),
$$
where
$\psi\in Q(\varphi)$ (see \cite{KleMes05}). When $\varphi(0)=\infty$, the t-norm is said to be \emph{strict} and is of the form
$$
\mathsf{T}^{\varphi}(x_1,x_2)=\varphi^{-1}(\varphi(x_1)+\varphi(x_2)).
$$
The mean $\M_{\mathsf{T}^{\varphi}}$ is then the \emph{quasi-arithmetic
mean}
$$
\M_{\varphi}(x_1,x_2)=\textstyle{\varphi^{-1}\big(\frac 12\varphi(x_1)+\frac 12\varphi(x_2)\big)}
$$
and we can write
$\mathsf{T}^{\varphi}=\delta_{\mathsf{T}^{\varphi}}\circ\M_{\varphi}$. When $\varphi(0)<\infty$, the t-norm is said to be \emph{nilpotent} and
is of the form
$$
\mathsf{T}^{\varphi}(x_1,x_2)=\varphi^{-1}\big(\Min(\varphi(0),\varphi(x_1)+\varphi(x_2))\big).
$$
In this case, the mean
$\M_{\mathsf{T}^{\varphi}}$ may be very difficult to calculate. However, using Proposition~\ref{prop:TransfF} with $\F$ being the sum function,
it is easy to see that the quasi-arithmetic mean $\M_{\varphi}$ is again a $\mathsf{T}^{\varphi}$-level mean so that we can write
$\mathsf{T}^{\varphi}=\delta_{\mathsf{T}^{\varphi}}\circ\M_{\varphi}$, with
$\delta_{\mathsf{T}^{\varphi}}(x)=\varphi^{-1}\big(\Min(\varphi(0),2\varphi(x))\big)$.
\end{example}

\subsection{Transformed continuous functions}

We now consider the problem of finding necessary and sufficient conditions on a given nondecreasing function $\F\colon\I^n\to\R$ for its
factorization as $\F=f\circ\G$, where $\G\colon\I^n\to\I$ is nondecreasing and continuous and $f\colon\ran(\G)\to\R$ is nondecreasing. Such a
function $\F$ is then continuous up to possible discontinuities of $f$.

The following result solves this problem when we further assume that $\G$ satisfies condition (\ref{eq:Cond1C}). The general case remains an
interesting open problem.

\begin{theorem}\label{thm:TrContFct}
Let $\F\colon\I^n\to\R$ be a nondecreasing function. The following assertions are equivalent:
\begin{enumerate}
\item[(i)] There is a nondecreasing and continuous function $\G\colon\I^n\to\I$, satisfying condition (\ref{eq:Cond1C}), and a nondecreasing
function $f\colon\ran(\G)\to\R$ such that $\F=f\circ\G$.

\item[(ii)] $\F$ satisfies conditions (\ref{eq:ranFrandF}), (\ref{eq:Cond1C}), (\ref{eq:condNEW1}), (\ref{eq:condNEW2}), (\ref{eq:Cond2C}), and
(\ref{eq:Cond3C}).
\end{enumerate}
If these conditions hold, then we can choose $\G=\M_{\F}$ and $f=\delta_{\F}$.
\end{theorem}

\begin{proof}
Let us prove that $(i)\Rightarrow (ii)$. By Corollary~\ref{cor:ContDisc3} we have that $\G=\delta_{\G}\circ\M_{\G}$, where $\M_{\G}$ is
continuous. If follows that $\F=f\circ\delta_{\G}\circ\M_{\G}=\delta_{\F}\circ\M_{\G}$ and hence $\F$ satisfies (\ref{eq:ranFrandF}). We then
conclude by Theorem~\ref{thm:MainCont}.

Let us prove that $(ii)\Rightarrow (i)$. By Theorem~\ref{thm:MainCont}, we have $\F=\delta_{\F}\circ\M_{\F}$, where $\M_{\F}$ is nondecreasing,
idempotent (hence $\G$ satisfies (\ref{eq:Cond1C})), and continuous.
\end{proof}

\begin{remark}
If we remove condition (\ref{eq:Cond1C}) from assertion $(i)$ of Theorem~\ref{thm:TrContFct}, then $\F$ still satisfies (\ref{eq:ranFrandF}) but
may or may not satisfy (\ref{eq:Cond1C}).
\end{remark}

\section*{Acknowledgments}

The author wishes to thank the reviewer for helpful comments and suggestions. This research is supported by the internal research project
F1R-MTH-PUL-09MRDO of the University of Luxembourg.

\appendix
\section{Proof of Theorem~\ref{thm:GF-ND-Id}}
\label{app:GF-ND-Id}

We first consider a definition and two lemmas.

A subset $C$ of $\I^n$ is said to be an \emph{upper subset} if for any $\bfx\in C$ and any $\bfx'\in \I^n$, with $\bfx\leqslant\bfx'$, we have
$\bfx'\in C$. To give an example, for every $y\in\ran(\delta_{\F})$, the upper level set $\F_{>}^{-1}(y)$ is an upper subset of $\I^n$.

\begin{lemma}\label{lemma:mCd}
Let $\bfx,\bfx'\in\I^n$, with $\bfx\leqslant\bfx'$, and let $C$ be a nonempty upper subset of $\I^n$. Then $d_{\infty}(\bfx,C)\geqslant
d_{\infty}(\bfx',C)$.
\end{lemma}

\begin{proof}
Denote by $C^*$ the smallest upper subset of $\R^n$ containing $C$.  For every $\bfz\in C^*$, we have
$d_{\infty}(\bfx,\bfz)=d_{\infty}(\bfx',\bfz+\bfx'-\bfx)$ and $\bfz+\bfx'-\bfx\in C^*$. It follows that $\{d_{\infty}(\bfx,\bfz):\bfz\in
C^*\}\subseteq \{d_{\infty}(\bfx',\bfz'):\bfz'\in C^*\}$ and hence $d_{\infty}(\bfx,C)=d_{\infty}(\bfx,C^*)\geqslant
d_{\infty}(\bfx',C^*)=d_{\infty}(\bfx',C)$.
\end{proof}

\begin{lemma}\label{lemma:NDlevelS}
Assume $\F\colon\I^n\to\R$ is nondecreasing. Then any solution $\G\colon\I^n\to\I$ of Chisini's equation (\ref{eq:FuncEqFG2}) is nondecreasing
if and only if it is nondecreasing on each level set of $\F$.
\end{lemma}

\begin{proof}
The necessity is trivial. For the sufficiency, assume that the solution $\G\colon\I^n\to\I$ is nondecreasing on each level set of $\F$. Let
$\bfx,\bfx'\in\I^n$ be such that $\bfx\leqslant\bfx'$ and $\F(\bfx)<\F(\bfx')$. By Proposition~\ref{prop:SolEqFDG}, we must have
$\G(\bfx)\in\delta_{\F}^{-1}\{\F(\bfx)\}$ and $\G(\bfx')\in\delta_{\F}^{-1}\{\F(\bfx')\}$. Therefore, since $\delta_{\F}$ is nondecreasing, we
also have $\G(\bfx)<\G(\bfx')$.
\end{proof}

\begin{proof}[Proof of Theorem~\ref{thm:GF-ND-Id}]
Let us first prove that $\F=\delta_{\F}\circ\M_{\F}$ or, equivalently, that $\M_{\F}(\bfx)\in\delta_{\F}^{-1}\{\F(\bfx)\}$ for all $\bfx\in\I^n$
(see Proposition~\ref{prop:SolEqFDG}). Fix $\bfx^*\in\I^n$. By definition of $\M_{\F}$, we always have
$\M_{\F}(\bfx^*)\in\left[a_{\F}(\bfx^*),b_{\F}(\bfx^*)\right]$. We now have to prove that if
$a_{\F}(\bfx^*)\notin\delta_{\F}^{-1}\{\F(\bfx^*)\}$ (which implies $a_{\F}(\bfx^*) < b_{\F}(\bfx^*)$) then necessarily
$\M_{\F}(\bfx^*)>a_{\F}(\bfx^*)$. For the sake of contradiction, suppose that $\M_{\F}(\bfx^*)=a_{\F}(\bfx^*)$.
\begin{enumerate}
\item[(i)] If $\bfx^*\in\Omega_{\F}$ then $d_{\F}^<(\bfx^*)=0$ and hence condition (\ref{eq:Case1N}) holds. It then follows that
$\tilde{d}_{\F}^<(\bfx^*)=0$, that is $a_{\F}(\bfx^*)=\Max(\bfx^*)$. This implies $\bfx^*\leqslant a_{\F}(\bfx^*)\mathbf{1}$ and hence
$\F(\bfx^*)\leqslant \delta_{\F}(a_{\F}(\bfx^*))<\F(\bfx^*)$, a contradiction.

\item[(ii)] If $\bfx^*\notin\Omega_{\F}$ then at least one of the conditions (\ref{eq:Case1N}) and (\ref{eq:Case2N}) must hold. This implies
$a_{\F}(\bfx^*)=\Max(\bfx^*)$, again a contradiction.
\end{enumerate}
The case when $b_{\F}(\bfx)\notin\delta_{\F}^{-1}\{\F(\bfx)\}$ can be dealt with dually.

Let us now prove that $\M_{\F}$ is nondecreasing. By Lemma~\ref{lemma:NDlevelS} we only need to prove that $\M_{\F}$ is nondecreasing on each
level set of $\F$. Fix $\bfx^*\in\I^n$ and let $\bfx,\bfx'\in\F^{-1}\{\F(\bfx^*)\}$, with $\bfx\leqslant\bfx'$. We only need to show that
$\M_{\F}(\bfx)\leqslant\M_{\F}(\bfx')$.

If the set $\F_{>}^{-1}(\F(\bfx^*))$ is nonempty (which means that $d_{\F}^>(\bfx^*)<\infty$), then it is an upper subset of $\I^n$ and, by
Lemma~\ref{lemma:mCd}, we must have $d_{\F}^>(\bfx) \geqslant d_{\F}^>(\bfx')$ and $\tilde{d}_{\F}^>(\bfx) \geqslant \tilde{d}_{\F}^>(\bfx')$,
and we prove dually that $d_{\F}^<(\bfx) \leqslant d_{\F}^<(\bfx')$ and $\tilde{d}_{\F}^<(\bfx) \leqslant \tilde{d}_{\F}^<(\bfx')$. We can now
assume without loss of generality that $\F_{<}^{-1}(\F(\bfx^*))$ and $\F_{>}^{-1}(\F(\bfx^*))$ are nonempty. Assume also that conditions
(\ref{eq:Case1N}) and (\ref{eq:Case2N}) do not hold. Four exclusive cases are to be examined:
\begin{enumerate}
\item[(i)] If $\bfx,\bfx'\in\Omega_{\F}$ then, assuming $d_{\F}^<(\bfx)> 0$, we have
$$
\M_{\F}(\bfx) = a_{\F}(\bfx)+\frac{b_{\F}(\bfx)-a_{\F}(\bfx)}{\frac{d_{\F}^>(\bfx)}{d_{\F}^<(\bfx)}+1} \leqslant
a_{\F}(\bfx)+\frac{b_{\F}(\bfx)-a_{\F}(\bfx)}{\frac{d_{\F}^>(\bfx')}{d_{\F}^<(\bfx')}+1}=\M_{\F}(\bfx').
$$
If $d_{\F}^<(\bfx)=0$ then we simply have $\M_{\F}(\bfx)=a_{\F}(\bfx)=a_{\F}(\bfx')\leqslant\M_{\F}(\bfx')$.

\item[(ii)] If $\bfx,\bfx'\in\I^n\setminus\Omega_{\F}$ then $\M_{\F}(\bfx)=\frac 12 a_{\F}(\bfx)+\frac 12 b_{\F}(\bfx)=\frac 12
a_{\F}(\bfx')+\frac 12 b_{\F}(\bfx')=\M_{\F}(\bfx')$.

\item[(iii)] If $\bfx\in\Omega_{\F}$ and $\bfx'\in\I^n\setminus\Omega_{\F}$ then $d_{\F}^<(\bfx)\leqslant d_{\F}^<(\bfx')=0$ and hence
$d_{\F}^<(\bfx)=0$. Therefore, $\M_{\F}(\bfx)=a_{\F}(\bfx)=a_{\F}(\bfx')\leqslant\M_{\F}(\bfx')$.

\item[(iv)] If $\bfx\in\I^n\setminus\Omega_{\F}$ and $\bfx'\in\Omega_{\F}$ then, similarly to the previous case, we must have
$d_{\F}^>(\bfx')=0$ and hence $\M_{\F}(\bfx')=b_{\F}(\bfx')=b_{\F}(\bfx)\geqslant\M_{\F}(\bfx)$.
\end{enumerate}
The situation when any of the conditions (\ref{eq:Case1N}) and (\ref{eq:Case2N}) hold can be dealt with similarly as in case (i) above.

Let us now prove that $\M_{\F}$ is idempotent. Let $x\mathbf{1}\in\mathrm{diag}(\Omega_{\F})$. Again, we can assume that
$\F_{<}^{-1}(\F(x\mathbf{1}))$ and $\F_{>}^{-1}(\F(x\mathbf{1}))$ are nonempty. Then
$d_{\F}^<(x\mathbf{1})=\tilde{d}_{\F}^<(x\mathbf{1})=x-a_{\F}(x\mathbf{1})$ and
$d_{\F}^>(x\mathbf{1})=\tilde{d}_{\F}^>(x\mathbf{1})=b_{\F}(x\mathbf{1})-x$ and hence $\M_{\F}(x\mathbf{1})=x$. Now, let
$x\mathbf{1}\in\mathrm{diag}(\I^n\setminus\Omega_{\F})$, which means that $d_{\F}^<(x\mathbf{1})=d_{\F}^>(x\mathbf{1})=0$. Then
$\delta_{\F}^{-1}\{\F(x\mathbf{1})\}=\delta_{\F}^{-1}\{\delta_{\F}(x)\}$ is the singleton $\{x\}$. Indeed, suppose on the contrary that
$\delta_{\F}(x')=\delta_{\F}(x)$ for some $x'>x$. Then $\F$ would be constant on $[x,x']^n$ and hence $d_{\F}^>(x\mathbf{1})>0$, a
contradiction. Therefore, $\M_{\F}(x\mathbf{1})=x$.
\end{proof}


\end{document}